\documentclass[a4paper, twoside, 11pt]{article}
\usepackage{amsmath,amssymb,amsthm,braket,nccmath}
\usepackage[margin=30mm]{geometry}
\usepackage{indentfirst}
\usepackage{fancyhdr}
\usepackage{comment}

\numberwithin{equation}{section}

\newcommand{\N}{\mathbb{N}}
\newcommand{\R}{\mathbb{R}}
\newcommand{\T}{\mathbb{T}}
\newcommand{\Z}{\mathbb{Z}}
\newcommand{\ep}{\varepsilon}
\newcommand{\bfx}{\mathbf{x}}

\newcommand{\f}{\frac}
\newcommand{\lmd}{\lambda}
\newcommand{\Lmd}{\Lambda}
\newcommand{\cA}{\mathcal{A}}
\newcommand{\cB}{\mathcal{B}}
\newcommand{\cD}{\mathcal{D}}
\newcommand{\cE}{\mathcal{E}}
\newcommand{\cF}{\mathcal{F}}
\newcommand{\cH}{\mathcal{H}}
\newcommand{\cI}{\mathcal{I}}
\newcommand{\cJ}{\mathcal{J}}
\newcommand{\cN}{\mathcal{N}}
\newcommand{\cO}{\mathcal{O}}

\newcommand{\cR}{\mathcal{R}}
\newcommand{\fA}{\mathfrak{A}}
\newcommand{\fS}{\mathfrak{S}}
\newcommand{\psiep}{\psi^\ep}
\newcommand{\phiep}{\phi^\ep}

\newcommand{\tpsi}{\tilde{\psi}^\ep}

\newcommand{\tW}{\tilde{W}}
\newcommand{\Fav}{F_{\text{av}}}

\newcommand{\h}{\f{1}{2}}
\newcommand{\al}{\alpha}

\newcommand{\s}{\sigma}
\newcommand{\om}{\omega}
\newcommand{\e}{\eta}
\newcommand{\ka}{\kappa}
\newcommand{\z}{\zeta}
\newcommand{\Tmax}{T_{\text{max}}}

\newtheorem*{theorem*}{Theorem}
\newtheorem{thm}{Theorem}[section]

\newtheorem*{definition*}{Definition}
\newtheorem{lemma}[thm]{Lemma}
\newtheorem*{lemma*}{Lemma}
\newtheorem{prop}[thm]{Proposition}
\newtheorem{remark}[thm]{Remark}
\newtheorem{cor}[thm]{Corollary}

\title{Global approximation for the cubic NLS with strong magnetic confinement}
\author{\small{JUMPEI KAWAKAMI}
}
\date{}

\begin{document}
\markboth{\centerline{\footnotesize{JUMPEI KAWAKAMI}}}{\centerline{\footnotesize{Global approximation for the cubic NLS with strong magnetic confinement}}}

\maketitle
\renewcommand{\thefootnote}{\fnsymbol{footnote}}
\renewcommand{\thefootnote}{\arabic{footnote}}
\begin{abstract}
We consider nonlinear Schr\"{o}dinger equation with strong magnetic fields in 3D. This model was derived by R L.~Frank, F.~M\'{e}hats, C.~Sparber in 2017. We prove modified scattering for small initial data and the existence of modified wave operator for small final data. To describe asymptotic behavior of the NLS we use the time-averaged model which was derived by the same authors as "the strong magnetic confinement limit" of the NLS. We construct asymptotic solutions which satisfy both asymptotic in time evolution and convergence in the strong magnetic confinement limit. We also analyze the error between the solution to the NLS and the time-averaged model for the same initial data and obtain global estimates.
\end{abstract}

\section{Introduction}\label{intro}
\subsection{The derivation of NLS-type models}
We consider nonlinear Schr\"{o}dinger equation given by
\begin{equation}\label{ep}
i\partial_t \psiep=\f{1}{\ep^2}\cH \psiep-\f{1}{2}\partial_z^2 \psiep +V(z)\psiep +\lmd |\psiep|^{2\sigma} \psiep \tag{$\ep$-NLS}
\end{equation}
 where $(t,\mathbf{x})\in \R \times \R^3 $, $\lmd=\pm1$, $\sigma \in \mathbb{N}$ and $V$ is a real valued, smooth and subquadratic potential.  We will denote the spatial variables by $\bfx=(x_1,x_2,z)$ and  $x=(x_1,x_2)$.  The parameter $0<\ep \ll 1$ implies the strength of a constant magnetic field. The operator $\cH$ is  the classical Landau Hamiltonian in symmetric gauge:
\[\cH=\f{1}{2}(-i\nabla_x+\f{1}{2}x^\perp)^2=-\f{1}{2}\Delta_x+\f{1}{8}|x|^2-\f{i}{2}x^\perp\cdot\nabla_x , \]
where  $x^\perp=(-x_2,x_1)$.
This model was derived in \cite{Fr} to analyze "the strong magnetic confinement limit" as $\ep \to +0$ for the following NLS:

\begin{equation}\label{eq1.1}
	i\partial_{t}\psi= \frac{1}{2}(-i\nabla_{\bfx}+A^{\ep}(\bfx))^2\psi +V(z)\psi +\beta^{\ep} |\psi|^{2\sigma}\psi,
\end{equation}
which arise e.g. in \cite{BEC}. Here, $\beta^\ep \in \mathbb{R}$ and the vector potential $A^\ep$ is assumed to be given by 
\[ A^{\ep}(\mathbf{x})=\frac{1}{2\ep^2}(-x_2,x_1,0). \] 
The corresponding magnetic field is given by
\[ B^{\ep}=\nabla \times A^{\ep}=\frac{1}{\ep^2}(0,0,1) \in \R^3,\] 
namely, a constant magnetic field in the $z$-direction with field strength $|B^\ep|=\f{1}{\ep^2} \gg 1$.\\

 We will follow the ountline of the derivation of the limiting model in \cite[Section 1]{Fr}.
The operator $\cH$ is essentially self-adjoint on $C_0^\infty(\R^2) \subset L^2(\R^2)$ with pure point spectrum given by
\begin{equation}\label{specn}
\text{spec}\cH=\{n+\h : n \in \N_0\}.
\end{equation}
In view of (\ref{ep}), $\cH$ induces high frequency oscillations in time  within the solution $\psiep$. By considering the profile  filtering these oscillations, we expect the following limit in a strong norm, 
\begin{equation}\label{smc}
  \phiep(t,\bfx):=e^{it\cH/\ep^2}\psiep(t,\bfx)\longrightarrow\phi(t,\bfx) \quad \text{as}\quad \ep \to +0 .
\end{equation}
Then \cite{Fr} derived the limiting model as $\ep\to +0$ formally, which describes the averaged particle dynamics:
\begin{equation}\label{lmt}
i\partial_t \phi = -\h\partial_z^2\phi+V(z)\phi+\lmd \Fav(\phi) \tag{L-NLS}
\end{equation}
\begin{equation}\label{Fav}
\Fav(u)=\f{1}{2\pi}\int_0^{2\pi} e^{i\theta\cH}(|e^{-i\theta\cH}u|^{2\sigma}e^{-i\theta\cH}u)d\theta .
\end{equation}
Averaging of the nonlinearity is consistent with extracting the resonant terms of $\phiep$ in the $x$-direction, cf.~\cite{CR}. 

\subsection{Previous studies}

NLS with strong anisotropic electric or magnetic potentials have widely studied e.g. in \cite{N.Ben2, N.Ben, N.Ben3, Fr, K}. Especially,
\cite{Fr}  studied \eqref{ep} and \eqref{lmt} in 
\[ \Sigma^2=\{ u \in H^2(\R^3): |\bfx|^2u \in L^2(\R^3)  \} ,\]
which is a Banach algebra and make it easy to treat the case with all $\sigma\in \N$. They proved the local well-posedness of (\ref{lmt}) and the locally uniformly convergence of (\ref{smc}) in the $L^2$ topology.   \cite{K} studied these models in 
\[ \Sigma^1=\{ u \in H^1(\R^3): |\bfx|u \in L^2(\R^3) \}.\]
Because $\Sigma^1$ solutions to \eqref{lmt} and \eqref{ep} can be controlled by the conservative quantities of each model, \cite{K} proved global well posedness of the two models in $\Sigma^1$ with some conditions  and improved the convergence results in \cite{Fr}. \cite{K} proved the limit \eqref{smc} holds in the $\Sigma^1$ topology and uniformly in any arbitrarily long finite time interval.\\

The time averaged model \eqref{lmt} is similar structure to the resonant system of NLS with harmonic trapping in the $x$-direction. The resonant system (CR) is derived in \cite{LBL} and studied e.g. in \cite{CR, CR2, CRH} and used e.g. in \cite{MS} as a tool to analyze dynamics of the cubic NLS with partial harmonic oscillator.  \\

Global dynamics of NLS with partial harmonic oscillator is also  studied e.g. in \cite{Scpht, Gdb}.
\subsection{Our aim and strategies}

Our aim is to justify the approximation \eqref{smc} globally and to describe asymptotic behavior of \eqref{ep}. For this aim, we assume $V\equiv0$ and $\sigma=1$, which is the most basic and physically important case. In this case we get 1-dimensional dispersive effect, hence we can use the strategy by \cite{MST, MS}. \cite{MST} studied asymptotic behavior of NLS on $\R\times \T^d$:
\begin{equation}\label{HT}
	 (i\partial_t + \Delta_{x,y})U=|U|^2U \qquad (x,y)\in \R \times \T^d,
\end{equation}
 and \cite{MS} treated the same problem for NLS with partial harmonic oscillator on $\R^{1+d}$:
\begin{equation}\label{onT}
	(i\partial_t -\Delta_{x,y} + |y|^2)U=\ka_0|U|^2U \qquad (x,y)\in \R \times \R^d, \quad  \ka_0=\pm1.
\end{equation}
 These studies used a different type of resonant system from \eqref{lmt}: 
\begin{equation}\label{fr0}
 	i\partial_t G=\cF_z^{-1}[\Fav(\hat{G})]
\end{equation}
where $\hat{\cdot}$ is the Fourier transform and $\cF_z^{-1}$ is the inverse Fourier transform in $z$.
 This system extracts the resonance not only in the $x$ but also in the  $z$-directions (See \cite{MST,MS} and Section \ref{frs} in this paper). We use this resonant system to get a priori estimates for \eqref{ep} and \eqref{lmt} but did not use to describe the asymptotic dynamics of \eqref{ep} because this system has the following merits and demerits.\\

 First, an advantage of this resonant system is that we have good estimates in the norms used in \cite{MS} and this paper. This property is very useful to obtain a priori estimates for solutions to \eqref{HT} and \eqref{onT}. Second, if we use this resonant system as the leading term of NLS as above, we can reduce the study of the NLS to the analysis for the resonant system only on  $y\in \R^d$ or $\T^d$ in a sense, c.f \cite[Section 4.2]{MST} and \cite[Section 3.2]{MS}. These benefits hold true in our case.
  
 On the other hand, we are interested in the strong magnetic confinement limit as $\ep\to+0$.  But \eqref{fr0} extracts the resonance in the $z$-direction, so if we use this system as the leading term of \eqref{ep}, An error independent of $\ep$ will remain. \\
 
 Moreover, \eqref{lmt} has good properties regarding symmetries, conservation laws, etc.(cf. \cite{CR, K}), so we expect the study of the approximation of \eqref{ep} using \eqref{lmt} contributes to elucidation of dynamics of NLS with strong magnetic confinement. 
 For these reasons, we follow \cite{Fr, K} and use \eqref{lmt} as the leading term of \eqref{ep}.\\

 We note that if we dispose of the convergence property as $\ep\to+0$, we have the same results as in \cite{MS}. That is, we can describe the dynamics of \eqref{ep} and \eqref{lmt} by \eqref{fr0}. This is Proposition \ref{main3} in this paper.

\subsection{Main results}

We study \eqref{ep} with the assumptions $\sigma=1$ and $V(z)\equiv0$, i.e.
\begin{equation}\label{ep0}
i\partial_t \psiep=\f{1}{\ep^2}\cH \psiep-\f{1}{2}\partial_z^2 \psiep +\lmd |\psiep|^{2} \psiep  \tag{$\ep$-NLS$_0$}.
\end{equation}
So we impose the same assumptions to \eqref{lmt}:
\begin{equation}\label{lmt0}
i\partial_t \phi = -\h\partial_z^2\phi+\f{\lmd}{2\pi}\int_0^{2\pi}
e^{i\theta\cH}(|e^{-i\theta\cH}\phi|^{2}e^{-i\theta\cH}\phi)d \theta \tag{L-NLS$_0$}.
\end{equation}
From now on, the notation $\Fav$ implies \eqref{Fav} with $\sigma=1$. Moreover, for the sake of simplicity, we assume $\lmd=1$.
In this case we expect the existence of a modified wave operator for final data and formulation of modified scattering for initial data because these models have the cubic nonlinearity and $1$-dimensional dispersive effect $t^{-1/2}$ in the $z$-direction. Hence we have the following results. The definition of $S$ and  $S^+$norms is in Section \ref{othernorms}.  \\

Our first result is the existence of modified wave operator for small final data.

 \begin{thm}\label{main1}
 There exist $\alpha>0$ and $\ep_1=\ep_1(\al)\in (0,1)$ such that the following statement holds. \\
 
 \noindent
 If $\phi_0\in S^+$ satisfies 
 \[ \|\phi_0\|_{S^+}\le \al, \]
 then there exist a global solution to  \eqref{lmt0} $\phi \in C([0,\infty):S^+)$  with $\phi(0)=\phi_0$ and global solutions to \eqref{ep0} $\{\psiep\}_{0<\ep<1} \subset  C([0,\infty):S^+) $ such that 
 \begin{equation}\label{main1-1}
 \|e^{-it\partial^2_z/2}e^{it\cH/\ep^2}\psiep(t)- e^{-it\partial_z^2/2}\phi(t)\|_{S^+}\le C \ep^{\f{1}{20}}(1+t)^{-\f{1}{20}}
 \end{equation}
 holds for all  $t\in [1,\infty)$ and $\ep\in (0,1)$. Moreover it holds
  \begin{equation}\label{main1-2}
 	\|e^{-it\partial^2_z/2}e^{it\cH/\ep^2}\psiep(t)- e^{-it\partial_z^2/2}\phi(t)\|_{S}\le C_\al \ep^{\f{1}{20}}(1+t)^{-\f{1}{20}}
 \end{equation}
  for all  $t\in [0,\infty)$ and $\ep\in (0,\ep_1)$. Here, $C_\al$ is a constant depending on $\al$.
 \end{thm}
 
 Theorem \ref{main1}  claims that for any small solution to \eqref{lmt0} there exist asymptotic solutions to \eqref{ep0} which move in closer to the solution to \eqref{lmt0} as $ t\to +\infty$ and approach uniformity in time as $\ep \to +0$. Note that  $\phi_0 = \psiep(0)$ does not necessarily hold.\\ 

 Our second result is the formulation of modified scattering  and estimates of the limit as $\ep \to +0$ for small solutions. 
  \begin{thm}\label{main2}
 
  There exist $\alpha>0$ and $\ep_1=\ep_1(\al)\in (0,1)$ such that the following statement holds. \\
 
 \noindent
 If $\psi_0\in S^+$ satisfies 
 \[ \|\psi_0\|_{S^+}\le \al, \]
 then for any $\ep\in (0,1)$ a global solution to  \eqref{ep0} $\psiep \in C([0,\infty):S^+)$  with $\psiep(0)=\psi_0$ exists and satisfies the following properties.
 \begin{enumerate}
 	\item
	There exists a solution to \eqref{lmt0} $\phiep \in C([0,\infty):S^+) $ such that 
 	\begin{equation}
 		\|e^{-it\partial^2_z/2}e^{it\cH/\ep^2}\psiep(t)- e^{-it\partial_z^2/2}\phiep(t)\|_{S^+}\le C_\al \ep^{\f{1}{20}}(1+t)^{-\f{1}{20}}
 	\end{equation}
 	holds  for all  $t\in [1,\infty)$ and $\ep\in (0,1)$, and 
 		\begin{equation}
 		\|e^{-it\partial^2_z/2}e^{it\cH/\ep^2}\psiep(t)- e^{-it\partial_z^2/2}\phiep(t)\|_{S}\le C_\al \ep^{\f{1}{20}}(1+t)^{-\f{1}{20}}
 	\end{equation}
 	holds  for all  $t\in [0,\infty)$ and $\ep\in (0,\ep_1)$.
 	
 	\item
 	There exist a global solution to \eqref{lmt0} $\phi \in C([0,\infty):S^+)$ with $\phi(0)=\psi_0$  and $\ep_0>0$ depending on $\psi_0$ such that 
 		\begin{equation}\label{main2-1}
 		\|e^{-it\partial^2_z/2}e^{it\cH/\ep^2}\psiep(t)- e^{-it\partial^2_z/2}\phi(t)\|_{S^+}\le C_{\ep, \psi_0} (1+t)^{C\al^2}, \quad \lim_{\ep\to+0}C_{\ep,  \psi_0}=0
 	\end{equation}
 hold for all $t\in [0,\infty)$ and $\ep\in (0,\ep_0]$. Moreover,  
 	\begin{equation}\label{main2-2}
 		\|e^{-it\partial^2_z/2}e^{it\cH/\ep^2}\psiep(t)- e^{-it\partial^2_z/2}\phi(t)\|_{S}\le C_\al \ep^{\f{1}{20}}(1+t)^{C\al^2},
 	\end{equation}
 holds for  all $t\in [0,\infty)$ and $\ep\in (0,\ep_1]$.
 \end{enumerate}
 Here, $C_\al$ is a constant depending on $\al$ and $C_{\ep,  \psi_0}$ is a constant depending on $\ep$ and $\psi_0$. 
 \end{thm}
Theorem \ref{main2} $(i)$ claims that asymptotic behavior of the small solution $\psiep$ to \eqref{ep0} can be described by the solution  $\phiep$ to \eqref{lmt0} which moves in closer to $\psiep$ as $t\to +\infty$ and approaches uniformity in time as $\ep \to +0$. This is good approximation, but $\psi_0 = \phiep(0)$ does not necessarily hold. On the other hand in $(ii)$, the solution $\phi$ to \eqref{lmt0} has the same initial data as $\psiep$ and we consider the error between $\psiep$ and $\phi$. This is the same situation as \cite{Fr, K}. However, unlike \cite{Fr, K} we impose the smallness condition and obtain the global estimates in time. \\

With respect to Theorem \ref{main1} and Theorem \ref{main2} $(i)$, we can prove that small solutions to \eqref{ep0} and \eqref{lmt0} do not scatter. 

\begin{prop}\label{main3}
There exists $\alpha>0$ such that the following statement holds. \\
	  
\noindent
\begin{enumerate}

\item 
If $\psi_0\in S^+$ satisfies 
\[ \|\psi_0\|_{S^+}\le \al, \]
then for any $\ep\in (0,1)$ a global solution to  \eqref{ep0} $\psiep \in C([0,\infty):S^+)$  with $\psiep(0)=\psi_0$ exists and exibits modified scattering to \eqref{fr0}{\rm:} there exists a global solution to \eqref{fr0} $G^\ep(t) \in C([0,\infty),S)$ which satisfies
	\begin{equation*}
	\lim_{t\to\infty}\|e^{-it\partial^2_z/2}e^{it\cH/\ep^2}\psiep(t)- G^\ep(\pi \ln t)\|_{S} =0.
\end{equation*}

\item 
If $\phi_0\in S^+$ satisfies 
\[ \|\phi_0\|_{S^+}\le \al, \]
then a global solution to  \eqref{lmt} $\phi \in C([0,\infty):S^+)$  with $\phi(0)=\phi_0$ exists and exibits modified scattering to \eqref{fr0}{\rm:} there exists a global solution to \eqref{fr0}  $G(t)\in C([0,\infty),S)$ which satisfies
\begin{equation*}
	\lim_{t\to\infty}\|e^{-it\partial^2_z/2}\phi(t)- G(\pi \ln t)\|_{S} =0.
\end{equation*}
\end{enumerate}		  
\end{prop}
\begin{proof}
	Follow the proof of \cite[Theorem 1.2]{MS} in  \cite[Section 5.2]{MS}. In the proof we use  Proposition \ref{nlest2} in this paper. 
\end{proof}
	
\begin{cor}
	Let $F(t) \in C([0,\infty),S^+)$ denote either $e^{-it\partial^2_z/2}e^{it\cH/\ep^2}\psiep(t)$ or $e^{-it\partial^2_z/2}\phi(t)$, where $\psiep(t)$ is a solution to \eqref{ep0} and $\phi(t)$ is a solution to \eqref{lmt0}. If $F(t)$ converges as $t\to\infty$ and satisfies $\|F(0)\|_{S^+}\le \al$, where $\al$ is given in Proposition \ref{main3}, then $F(t)\equiv 0$.
\end{cor}
\begin{proof}
	See the proof of \cite[Corollary 1.4]{MS} in \cite[Section 3.3]{MS}.
\end{proof}

\subsection{Organization of the paper}
We introduce notations used in this paper in Section \ref{preli} and give the definition and some properties of two resonant systems in Section \ref{reson}. We prove the key estimates for the nonlinearity in Section \ref{struc} and the main results in Section \ref{proof}.

\section{Preliminaries}\label{preli}
\subsection{Notation}
In this section, we introduce some notations used in this paper. We denote functions on $\R^3_\bfx$  by capital letters and functions on $\R_x^2$ or $\R_z$  by lower case letters. We denote 
\[ \cH =\f{1}{2}(-i\nabla_x+\f{1}{2}x^\perp)^2=-\f{1}{2}\Delta_x+\f{1}{8}|x|^2-\f{i}{2}x^\perp\cdot\nabla_x  \]
and decompose into the Harmonic oscillator $\cH_0$  and the angular momentum operator $L$:
\[  \cH_0:=-\f{1}{2}\Delta_x+\f{1}{8}|x|^2, \qquad L:= -\f{i}{2}x^\perp\cdot\nabla_x  . \]
We use the notation
\begin{equation}
	\cD^\ep:=\f{1}{\ep^2}\cH-\h\partial_z^2, \qquad \cD_0^\ep:=\f{1}{\ep^2}\cH_0-\h\partial_z^2 .
\end{equation} 
Let us recall some properties of $\cH_0$ and $L$ (cf.\cite{CR}). These operators admit the same eigenfunctions $\{h_{n_1,n_2}\}_{n_1,n_2 \in \N_0}$, which satisfy
\begin{equation}\label{propH}
 \cH_0 h_{n_1,n_2}=\h(n_1+n_2+1) h_{n_1,n_2},  \qquad Lh_{n_1,n_2}=\h(n_1-n_2)h_{n_1,n_2},
\end{equation}
and form the Hilbertian basis of $L^2(\R^2)$. Furthermore,  we denote  $n$-th eigenspace of $\cH_0$ ($n=n_1+n_2$) by $E_n$ and the projection onto $E_n$ by $\Pi_n$. Then the Hermite expansion of $u\in L^2(\R^2)$ is denoted by
\[ u(x)=\sum_{n=0}^\infty u_n(x), \qquad u_n(x) := \Pi_nu(x)=\sum_{n_1,n_2\in \N_0, \text{ }   n_1+n_2=n} \langle u, h_{n_1,n_2} \rangle_{L^2} h_{n_1,n_2}. \] 

We define
\[ \|u\|_{\Sigma_x^s}^2:=\sum_{n\ge 0} \lmd_n^s \|u_n\|_{L_x^2}^2 , \quad s\ge 0,\]
where $\lmd_n=\h(n+1)$. For a bounded function $\varphi$, we define
\[ \varphi(H)u:=\sum_{n\ge0} \varphi(\lmd_n)u_n. \]
On the other hand we define the Fourier transform on $\R_z$ by
\[  \cF_z f(\z) = \hat{f}(\z) := \f{1}{2\pi} \int_\R e^{-iz\z} f(z) dz. \]  
 Similarly, for a function $F(x,z)$ on $\R_\bfx^3$, $\hat{F}(x,\z)$ denotes the partial Fourier transform in $z$. Then, the full frequency expansion of $F$ reads
 \[ F(x,z)=\sum_{n\in \N_0}\int_\R \hat{F}_n(x,\z)e^{i\z z} d\z  . \]
 We use the Littlewood-Paley decomposition. Let $\varphi \in C_0^\infty(\R)$, $\varphi(x)=1$ when $|x|\le1$ and  $\varphi(x)=0$ when $|x|>2$, we define the Littlewood-Paley projection $P_{\le N}$ ($N\in 2^{\N_0}$) as
\[ (\cF P_{\le N}F ) (x,\z) := \varphi\Big(\f{\cH_0}{N^2}\Big)\varphi\Big(\f{\z}{N}\Big)\hat{F}(x,\z), \]
and also define 
\[ P_N:= P_{\le N}-P_{\le N/2}, \qquad P_{\ge N}:= 1-P_{\le N/2}. \]
$\Sigma_0^s(\R^3)$ is the partial Hermite Sobolev space whose functions satisfy
\[ \|F\|_{\Sigma_0^s(\R^3)}:=\big( \sum_{N\in 2^{\N_0}} N^{2s}\|P_NF\|_{L^2(\R^3)}^2)^\h <\infty. \] 
From \cite{Yajima} we have the norm equivalence
\begin{equation} \label{normeq}
	\|F\|_{\Sigma_0^s(\R^3)}\simeq (\| F\|_{H^s(\R^3)}^2 + \||x|^sF\|_{L^2(\R^3)}^2)^\h. 
\end{equation}
We denote the Littlewood-Paley projection in $z$ only by $Q$, that is,
\[ (\cF Q_{\le N}F ) (x,\z) := \varphi\Big(\f{\z}{N}\Big)\hat{F}(x, \z), \]
\[ Q_N:= Q_{\le N}-Q_{\le N/2}, \qquad Q_{\ge N}:= 1-Q_{\le N/2}. \] 
By Parseval's identity, we have
\begin{equation}
\| [Q_N,z]\|_{L_z^2 \to L_z^2}\lesssim N^{-1}.
\end{equation}
 
For $\om\in\Z$, we define the following level sets:
\[\Gamma_\om:=\{ (p,q,r,s)\in \N_0^4 : \lmd_p-\lmd_q+\lmd_r-\lmd_s = \h(p-q+r-s)= \f{\om}{2} \}. \]
Especially, $\Gamma_0$ is the resonant level set.

\subsection{Other norms}\label{othernorms}
We will use three different norms introduced in \cite{MS}:

\begin{align*}
\|F\|_{Z}&:=\Big[\sup_{\z\in \R}(1+|\z|^2)^2\sum_{p}(1+p)\|\hat{F}_p(\cdot, \z)\|_{L_x^2}^2\Big]^{\f{1}{2}}  \simeq \Big[\sup_{\z\in \R}(1+|\z|^2)^2\|\hat{F}(\cdot, \z)\|_{\Sigma_x^1}^2\Big]^{\f{1}{2}}\\
\|F\|_{S}&:=\|F\|_{\Sigma_0^N}+\|zF\|_{L_\bfx^2} \\
\|F\|_{S^+}&:=\|F\|_{S}+\|(1-\partial_z^2)^4F\|_{S}+\|zF\|_{S}.
\end{align*}
The operator $P_{\le N}$, $Q_{\le N}$ and multiplication of $\varphi(\cdot /N)$ are bounded in $Z$, $S$, $S^+$ uniformly in $N$. From \cite[Lemma 2.1]{MS} it holds that
\begin{equation}\label{emb}
	 \|F\|_{\Sigma_0^1}\lesssim \|F\|_{Z}\lesssim \|F\|_{S} \lesssim \|F\|_{S^+}.
\end{equation}
We also use the following space-time norms: for fixed $0<\delta<10^{-4}$ and any $T\ge1$,  
\begin{align*}
\|F\|_{X_T}&:=\sup_{0\le t\le T}\{ \|F(t)\|_{Z}+(1+|t|)^{-\delta}\|F(t)\|_{S}+(1+|t|)^{1-3\delta}\|\partial_t F(t)\|_{S} \} \\
\|F\|_{X_T^+}&:=\|F\|_{X_T}+\sup_{0\le t\le T}\{ (1+|t|)^{-5\delta}\|F(t)\|_{S^+}+(1+|t|)^{1-7\delta}\|\partial_t F(t)\|_{S^+} \} 
\end{align*}
\begin{align*}
	\|F\|_{Y_T}&:=\sup_{T\le t}\{ \|F(t)\|_{Z}+(1+|t|)^{-\delta}\|F(t)\|_{S}+(1+|t|)^{1-3\delta}\|\partial_t F(t)\|_{S} \} \\
	\|F\|_{Y_T^+}&:=\|F\|_{Y_T}+\sup_{T\le t}\{ (1+|t|)^{-5\delta}\|F(t)\|_{S^+}+(1+|t|)^{1-7\delta}\|\partial_t F(t)\|_{S^+} \} .
\end{align*}

\subsection{Full nonlinearity}
Let us define the trilinear form $\cN^{t,\ep}$ by
\begin{equation}\label{nl}
\cN^{t,\ep}[F,G,H]:=e^{it\cD^\ep}(e^{-it\cD^\ep}F \cdot  \overline{e^{-it\cD^\ep}G} \cdot e^{-it\cD^\ep}H)
\end{equation}\label{n2}
where $ \ep\in (0,1)$. If there is no confusion, we omit the index $\ep$ : $\cN^t:=\cN^{t,\ep}$. 
We recall \eqref{propH} and the fact that the operator $e^{itL}$ is rotation of the $x$-coordinates of angle $t$, that is, for $x\in \R^2$ it holds
\[ e^{itL}f(x)=f(A_t x),\quad  A_t = \begin{pmatrix}
	\cos{t}  & -\sin{t} \\
	\sin{t} & \cos{t} \\
\end{pmatrix}
.\]
Then it holds
\begin{equation}\label{n3}
	\cN^{t}[F,G,H]=e^{it\cD_0^\ep}(e^{-it\cD_0^\ep}F \cdot e^{it\cD_0^\ep} \overline{G} \cdot e^{-it\cD_0^\ep}H).
\end{equation}\label{n4}
We denote $U^\ep$ by
\[ U^\ep(t,\bfx):=e^{it\cD^\ep}\psiep(t,\bfx). \]
We see that $\psiep$ solves \eqref{ep} if and  only if $U^\ep$ solves 
\begin{equation}\label{Uep}
i\partial_t U^\ep(t)=\cN^t[U^\ep(t), U^\ep(t), U^\ep(t)].
\end{equation}
In the Fourier side, the following formulation holds:
\[\cF_z\cN^t[F,G,H]=e^{it\f{\cH_0}{\ep^2}}\int_{\R^2}e^{it\e\kappa}e^{-it\f{\cH_0}{\ep^2}}\hat{F}(\z-\e)e^{it\f{\cH_0}{\ep^2}}\overline{\hat{G}(\z-\e-\kappa) } e^{-it\f{\cH_0}{\ep^2}}\hat{H}(\z-\kappa)d\e d\kappa.\]
Focusing on $z$-direction, we set
\begin{equation}\label{I}
\cI^t[f,g,h]:= e^{-it\f{\partial_z^2}{2}}(e^{it\f{\partial_z^2}{2}}f\cdot e^{-it\f{\partial_z^2}{2}} \overline{g} \cdot e^{it\partial_z^2}h).
\end{equation}
Then it holds
\[\cF_z\cI^t[f,g,h]=\int_{\R^2}e^{it\e\kappa}\hat{f}(\z-\e)\overline{\hat{g}(\z-\e-\kappa)} \hat{h}(\z-\kappa)d\e d\kappa,\]

\[\cF_z\cN^t[F,G,H]=e^{it\f{\cH_0}{\ep^2}} \cI^t [e^{-it\f{\cH_0}{\ep^2}}\hat{F}, e^{it\f{\cH_0}{\ep^2}}\overline{\hat{G}}, e^{-it\f{\cH_0}{\ep^2}}\hat{H}].\]
Using the Hermite expansion of $F$, $G$ and $H$, it also holds
\begin{equation}
\cF_z\cN^t[F,G,H]=\sum_{\om\in \Z} e^{it\f{\om}{2\ep^2}}\sum_{(p,q,r,s) \in \Gamma_\om} \Pi_p \int_{\R^2}e^{it\e\kappa} \hat{F_q}(\z-\e)\overline{\hat{G_r}(\z-\e-\kappa) } \hat{H_s}(\z-\kappa)d\e d\kappa
\end{equation}

\section{Resonant systems}\label{reson}
In this paper we consider two resonant systems,  ``partial'' resonant system and ``full'' resonant system. The partial resonant system extracts the resonance in the $x$-directions only and the $z$-direction remains as \eqref{nl}. The full resonant system extracts the resonance in both $x$ and $z$. So the error between the partial resonant system and \eqref{Uep} is smaller than between the full resonant system and \eqref{Uep}. On the other hand, the full resonant system matches $Z$, $S$ and $S^+$ norms better than the partial resonant system.
\subsection{Partial resonant system}\label{prs}
We define for $f,g,h \in L^2(\R_x^2)$ that
\[ \Fav(f,g,h):= \f{1}{2\pi}\int_{0}^{2\pi} e^{i\theta\cH}(e^{-i\theta\cH}f \cdot e^{i\theta\cH}\overline{g} \cdot e^{-i\theta\cH}h)d\theta. \]
By \eqref{Fav}, $\Fav(f)=\Fav(f,f,f)$. 
We introduce a trilinear form $\cN_0^t$ by
\begin{equation}\label{nlpr1}
\begin{split}
\cN_0^t[F,G,H]&:= e^{-it\f{\partial_z^2}{2}} \Fav(e^{it\f{\partial_z^2}{2}}F, e^{it\f{\partial_z^2}{2}}G, e^{it\f{\partial_z^2}{2}}H) \\
& =\f{1}{2\pi}\int_{0}^{2\pi} e^{i(\theta\cH-t\f{\partial_z^2}{2})}(| e^{-i(\theta\cH-t\f{\partial_z^2}{2})}\phi|^{2} e^{-i(\theta\cH-t\f{\partial_z^2}{2})}\phi)d\theta\\
& =\f{1}{2\pi}\int_{0}^{2\pi} e^{i(\theta\cH_0-t\f{\partial_z^2}{2})}(| e^{-i(\theta\cH_0-t\f{\partial_z^2}{2})}\phi|^{2} e^{-i(\theta\cH_0-t\f{\partial_z^2}{2})}\phi)d\theta.
\end{split}
\end{equation}
It also can be written as
\begin{equation}\label{nlpr2}
\cN_0^t[F,G,H]=\sum_{(p,q,r,s)\in \Gamma_0} \Pi_p \cN^t[F_q, G_r, H_s].
\end{equation}
We denote $W$ by
\[ W(t, \bfx):=e^{-it\f{\partial_z^2}{2}}\phi(t, \bfx) \]
where $\phi(t)$ is a solution to \eqref{lmt}.  Then $\phi$ solves \eqref{lmt} if and only if $W$ solves
\begin{equation}\label{pr}
 i\partial_t W(t) = \cN_0^t [W(t),W(t),W(t)]. \tag{PR}
\end{equation}
For the nonlinearity $\cN_0^t$, we have the following estimates.

\begin{lemma}[\cite{MS} Lemma 4.6]\label{N0}
Let $N\ge7$. Then we have

\begin{equation}\label{N01}
\|\cN_0^t[F^a, F^b,F^c]\|_{L_\bfx^2} \lesssim (1+|t|)^{-1} \min_{\sigma\in \fS^3}\|F^{\sigma(a)}\|_{L_\bfx^2}\|F^{\sigma(b)} \|_{Z_t^\delta} \|F^{\sigma(c)}\|_{Z_t^\delta} 
\end{equation}

\begin{equation}\label{N03}
\begin{split}
\|\cN_0^t[F^a, F^b,F^c]\|_{S} &\lesssim (1+|t|)^{-1} \max_{\sigma\in \fS^3}\|F^{\sigma(a)} \|_{S} \|F^{\sigma(b)} \|_{Z_t^\delta} \|F^{\sigma(c)}\|_{Z_t^\delta}	
\end{split}
\end{equation}

\begin{equation}\label{N04}
\begin{split}
\|\cN_0^t[F^a, F^b,F^c]\|_{S^+} \lesssim &(1+|t|)^{-1} \max_{\sigma\in \fS^3} \|F^{\sigma(a)}\|_{S^+} \|F^{\sigma(b)} \|_{Z_t^\delta} \|F^{\sigma(c)}\|_{Z_t^\delta}  \\
&+(1+|t|)^{-1+2\delta} \max_{\sigma\in \fS^3} \|F^{\sigma(a)} \|_{S} \|F^{\sigma(b)} \|_{S} \|F^{\sigma(c)}\|_{Z_t^\delta},
\end{split}
\end{equation}
where $\fS^3$ is the symmetric group of degree 3 and 
\[ \|F\|_{Z_t^\delta}:=\|F \|_{Z}+\langle t\rangle^{-\delta} \|F\|_{S}. \]

\end{lemma}

\begin{remark}
 \eqref{N01}  appears in the proof of {\rm\cite[Lemma 4.6]{MS}} but is not used.  In our case this is one of the key estimate.
\end{remark}

\subsection{Full resonant system}\label{frs}
We define the full resonant system
\begin{equation}\label{fr}
i\partial_t G(t) = \cR[G(t),G(t),G(t)] \tag{FR}
\end{equation}
where 
\begin{equation}\label{nlfr}
	\cR[F,G,H](\bfx):= \cF_z^{-1}[\Fav(\hat{F}, \hat{G}, \hat{H})](\bfx). 
\end{equation}
That is,
\begin{equation}
\begin{split}
i\partial_t \hat{G}(t,x,\z) &= \Fav(\hat{G}(t,x,\z)) \\
&=  \f{1}{2\pi}\int_{0}^{2\pi} e^{i\theta\cH}(|e^{-i\theta\cH}\hat{G}(t,x,\z)|^{2}e^{-i\theta\cH}\hat{G}(t,x,\z))d\theta .
\end{split}
\end{equation}
Referring to \cite[Section 3]{MS} (and \cite[Section 4]{MST}), we observe  the nonlinearity \eqref{nlfr} has the better estimates for $Z$, $S$, $S^+$ norms than \eqref{nlpr1}. This is the reason we need \eqref{pr} to obtain the a priori estimate, Proposition \ref{apriori}.

\section{Structure of the nonlinearity}\label{struc}
In this section, we extract the key effective interactions from the nonlinearity \eqref{nl}. We basically follow the strategy in \cite{MS}, but to fit our case, we need some modifications and to consider two kinds of decomposition.

\begin{prop}\label{nlest1}
	Fix $\ep \in(0,1). $We decompose \eqref{nl} as
	\begin{equation}
		\cN^{t,\ep}[F,G,H]=\cN_0^t[F,G,H]+\tilde{\cN}^{t,\ep}[F,G,H]
	\end{equation}
	where $\cN_0^t$ is given in \eqref{nlpr1}. Assume that, for $T^*\ge1$, $F$, $G$, $H: \R\to S$ satisfy
		\begin{equation}\label{asY}
		\|F\|_{Y_{T^*}} +\|G\|_{Y_{T^*}}+\|H\|_{Y_{T^*}} < \infty.
	\end{equation}
	Then the following estimate holds uniformly in $T^*\ge1$ and $\ep\in(0,1)${\rm :}
	\[ \sup_{T*\le T}T^{\f{1}{16}} \Big\| \int_{T/2}^T  \tilde{\cN}^{t,\ep}[F(t),G(t),H(t)]dt\Big\|_{S} \lesssim \ep^{\f{1}{13}} \|F\|_{Y_{T^*}} \|G\|_{Y_{T^*}} \|H\|_{Y_{T^*}} .\]
	Assuming in addition 
	\begin{equation}\label{asY+}
	\|F\|_{Y^+_{T^*}} +\|G\|_{Y^+_{T^*}}+\|H\|_{Y^+_{T^*}} < \infty,
	\end{equation}
	we also have
	\[ \sup_{T*\le T}T^{\f{1}{16}} \Big\| \int_{T/2}^T \tilde{\cN}^{t,\ep}[F(t),G(t),H(t)] dt\Big\|_{S^+} \lesssim \ep^{\f{1}{13}} \|F\|_{Y^+_{T^*}} \|G\|_{Y^+_{T^*}} \|H\|_{Y^+_{T^*}} .\]
	
\end{prop}

\quad\\
\begin{prop}[\cite{MS} Proposition 4.1]\label{nlest2}
Fix $\ep\in (0,1)$ and let $\cN_*^t$ denote $\cN^t$ or $\cN_0^t$. We decompose \eqref{nl} as
\begin{equation}
\cN_*^t[F,G,H]=\f{\pi}{t}\cR[F,G,H]+\cE^t[F,G,H]
\end{equation}
where $\cR$ is given in \eqref{nlfr}. Assume that, for $T^*\ge1$, $F$, $G$, $H: \R\to S$ satisfy 
	\begin{equation}\label{asX}
	\|F\|_{X_T^*} +\|G\|_{X_T^*}+\|H\|_{X_T^*}\le1.
\end{equation}
 Then we can write 
\[ \cE^t[F(t),G(t),H(t)]=\cE_1^t[F(t),G(t),H(t)]+ \cE_2^t[F(t),G(t),H(t)]=:\cE_1(t)+\cE_2(t)\]
and the following estimates hold uniformly in $T^*\ge1$ and $\ep\in (0,1)$ {\rm:}
\[\sup_{1\le T\le T^*}T^{-\delta} \Big\| \int_{T/2}^T \cE_i(t)dt\Big\|_{S} \lesssim 1, \qquad i=1,2,\]
\[ \sup_{1\le T\le T^*}(1+|t|)^{\f{34}{33}}\|\cE_1(t)\|_{Z} \lesssim 1 , \qquad \sup_{1\le T\le T^*}(1+|t|)^{\f{1}{10}}\|\cE_3(t)\|_{S}\lesssim 1,   \]
where $\cE_2(t)=\partial_t\cE_3(t)$. Assuming in addition
	\begin{equation}\label{asX+}
	\|F\|_{X^+_{T^*}} +\|G\|_{X^+_{T^*}}+\|H\|_{X^+_{T^*}}\le1,
\end{equation} 
we also have
\[ \sup_{1\le T\le T^*}T^{-5\delta} \Big\| \int_{T/2}^T \cE_i(t)dt \Big\|_{S^+} \lesssim 1, \quad  \sup_{1\le T\le T^*}T^{\f{34}{33}} \Big\| \int_{T/2}^T \cE_i(t)dt \Big\|_{S} \lesssim 1, \quad i=1,2.\]
\end{prop}
\quad \\

\begin{remark}
Comparing Proposition \ref{nlest2} with \cite[Proposition 4.1]{MS}, some estimates are improved. This is because  \cite[Proposition 4.1]{MS} claims  the minimum required rough estimates, and we can improve these estimates by following the proof of \cite[Proposition 4.1]{MS}. We need somewhat precise estimates to give $\al$ in Proposition \ref{apriori} without depending on $\delta$.

\end{remark}

To prove these propositions we have to get some preliminary estimates. For simplicity, in the following subsection we only mention the case with the assumptions \eqref{asX} and \eqref{asX+}.  But we also have the similar estimates corresponding to the case with the assumptions \eqref{asY} and \eqref{asY+}.

\subsection{The high-frequency parts}
We first consider estimates for high frequency in $z$.

\begin{lemma}\label{high}
Assume $N\ge 8$, fix $\ep\in(0,1)$ and let $\al\in (0,1]$ and $\cN_*^t$ denote $\cN^t$ or $\cN_0^t$. Then the following estimates hold uniformly in $T\ge1$, $\alpha$, and $\ep${\rm:}
\begin{equation}\label{h1}
\Big\| \sum_{\substack{A,B,C \in 2^{\N_0} \\ \max(A,B,C)\ge(T/\al)^{1/6}}} \cN_*^t[Q_AF,Q_BG,Q_CH] \Big\|_Z \lesssim \al^{-\f{7}{6}}T^{-\f{7}{6}}\|F\|_{S} \|G\|_S \|H\|_S,
\end{equation}
\begin{equation}\label{h2}
\begin{split}
\Big\| \sum_{\substack{A,B,C \in 2^{\N_0} \\ \max(A,B,C)\ge(T/\al)^{1/6}}} \int_{T/2}^T \cN_*^t[Q_AF(t),&Q_BG(t),Q_CH(t)]dt  \Big\|_{S} \\
& \lesssim \al^{\f{1}{13}}{T}^{-\f{1}{13}}\|F\|_{X_T} \|G\|_{X_T} \|H\|_{X_T},
\end{split}
\end{equation}
\begin{equation}\label{h3}
\begin{split}
\Big\| \sum_{\substack{A,B,C \in 2^{\N_0} \\ \max(A,B,C)\ge(T/\al)^{1/6}}} \int_{T/2}^T \cN_*^t[Q_AF(t),&Q_BG(t),Q_CH(t)]dt \Big\|_{S^+}  \\
&\lesssim \al^{\f{1}{13}}T^{-\f{1}{13}}\|F\|_{X^+_T} \|G\|_{X^+_T} \|H\|_{X^+_T}.
\end{split}
\end{equation}

\end{lemma}
\quad
\begin{proof}
The proof of \eqref{h1} is the same as \cite[Lemma 4.2]{MS}. We prove \eqref{h2}. We denote by “$\text{med}(A,B,C)$” the second largest dyadic number along $(A,B,C)$ and define
\[ \Lmd:=\{(A,B,C)\in 2^{\N_0} : \max(A,B,C)\ge \Big(\f{T}{\al}\Big)^{\f{1}{6}} \text{ and } \text{med}(A,B,C)\le \Big(\f{T}{\al}\Big)^{\f{1}{6}}/64 \}.\]
The case when $(A,B,C) \notin \Lmd$ is treated as in \cite{MST}, so we omit. \\

\noindent
We decompose 
\[ [ T/2, T] = \bigcup_{j\in J} I_j ,\quad  I_j =[j\al^{\f{1}{12}}T^{\f{11}{12}}, (j+1)\al^{\f{1}{12}}T^{\f{11}{12}}] \cap [T/2,T]=[t_j,t_{j+1}],\]
 where $\#J \lesssim T^{\f{1}{12}} \al^{-\f{1}{12}} $ holds.
Then, 
\[ \Big\| \sum_{(A,B,C)\in \Lmd} \int_{T/2}^T \cN_*^t[Q_AF(t),Q_BG(t),Q_CH(t)]dt \Big\|_{S} \lesssim E_1+E_2\]
where
\begin{equation*}
\begin{split}
E_1:=\Big\|  \sum_{j\in J} \sum_{(A,B,C)\in \Lmd} \int_{I_j}
 \Big(\cN_*^t[Q_AF(t)&,Q_BG(t),Q_CH(t)] \\
 & -\cN_*^t[Q_AF(t_j),Q_BG(t_j),Q_CH(t_j)]\Big) dt\Big \|_S
\end{split}
\end{equation*} 
\[ E_2:=  \sum_{j\in J} \sum_{(A,B,C)\in \Lmd} \Big\|\int_{I_j}   \cN_*^t[Q_AF(t_j),Q_BG(t_j),Q_CH(t_j)] dt\Big \|_S.\]
$E_1$ is bounded as
\begin{equation}\label{E1}
 E_1\le \sum_{j\in J} \int_{I_j}   E_{1,j}(t) dt
\end{equation}
where
\begin{equation*}
\begin{split}
E_{1,j}:= \Big\| \sum_{(A,B,C)\in \Lmd} \Big (\cN_*^t[Q_AF(t),&Q_BG(t),Q_CH(t)]\\
&-\cN_*^t[Q_AF(t_j),Q_BG(t_j),Q_CH(t_j)]\Big)  \Big \|_S.
\end{split}
\end{equation*} 
We can write
\begin{equation*}
\begin{split}
&\sum_{(A,B,C)\in \Lmd} \cN_*^t[Q_AF,Q_BG,Q_CH] \\
&= \cN_*^t[Q_{+}F,Q_{-}G,Q_{-}H] +   \cN_*^t[Q_{-}F,Q_{+}G,Q_{-}H]  +  \cN_*^t[Q_{-}F,Q_{-}G,Q_{+}H] 
\end{split}
\end{equation*} 
where $Q_+:= Q_{\ge(\f{T}{\al})^{\f{1}{6}}}$, $Q_-:= Q_{\le(\f{T}{\al})^{\f{1}{6}}/64}$. Using \cite[Lemma 2.2]{MS} or Lemma \ref{N0} with \eqref{emb}, and the boundedness of $Q_{\pm}$ on $S$, we have
 \begin{equation*}
\begin{split}
E_{1,j} \lesssim (1+|t|)^{-1} [& \|F(t)-F(t_j)\|_{S} \|G(t)\|_{S} \| H(t)\|_{S} \\
&+ \|F(t_j)|\|_{S} \|G(t)-G(t_j)\|_{S} \|H(t)\|_{S} \\
&+\|F(t_j)\|_{S} \|G(t_j)\|_{S} \|H(t)-H(t_j)\|_{S}].
\end{split}
\end{equation*} 
Since $|t-t_j|\le \al^{\f{1}{12}}T^{\f{11}{12}}$ holds for any $t \in I_j$, we see from the definition of $X_T$ norm that
\[ \|F(t)-F(t_j)\|_{S} \le \int_{t_j}^t \|\partial_t F(s)\|_{S} ds \lesssim \al^{\f{1}{12}}T^{-\f{1}{12}+3\delta}.\]
Similar bounds hold for $G$ and $H$. Hence we have
\[ E_{1,j} \lesssim \al^{\f{1}{12}}T^{-\f{13}{12}+5\delta} \|F\|_{X_T} \|G\|_{X_T} \|H\|_{X_T}. \] 
Applying this bound to \eqref{E1}, the contribution of $E_1$ is acceptable:
\[E_1 \lesssim \al^{\f{1}{12}}T^{-\f{1}{12}+5\delta} \|F\|_{X_T} \|G\|_{X_T} \|H\|_{X_T} . \]

\noindent
For $E_2$, by \cite[Lemma 6.1]{MS} we only have to estimate 
\[ \Big\|\int_{I_j}   \cN_*^t[Q_AF^a(t_j),Q_BF^b(t_j),Q_CF^c(t_j)] dt\Big \|_{L_{\bfx}^2}. \]
We use duality. Let $K \in L_\bfx^2$ and $F^a, F^b, F^c \in S$, 
\begin{equation*}
\begin{split}
\cA_K:=\Big\langle K,   \int_{I_j}  \cN_*^t[Q_AF^a,Q_BF^b,Q_CF^c] dt \Big\rangle_{L_{x,z}^2} .
\end{split}
\end{equation*}
In the  $\cN_*^t=\cN^t$ case, 
\begin{equation*}
	\begin{split}
\cA_K=& \int_{I_j}  \int_{\R^3} (e^{-it\cD_0^\ep}Q_AF^a) \overline{(e^{-it\cD_0^\ep}Q_BF^b)}(e^{-it\cD_0^\ep}Q_CF^c) \overline{(e^{-it\cD_0^\ep}K)} dxdzdt.
\end{split}
\end{equation*}
By Plancherel's theorem, we may assume that $K=Q_DK$, $D\simeq \max(A,B,C)$. Using H\"{o}lder's inequality and Lemma \ref{bi2} we have
\[ \cA_K \lesssim (\max(A,B,C))^{-1}\min_{\sigma \in \fS^3}\|F^{\sigma(a)}\|_{L_\bfx^2} \|F^{\sigma(b)}\|_{S}\|F^{\sigma(c)}\|_{S} .\]
Therefore
\begin{equation*}
\begin{split}
E_2&\lesssim  \sum_{j\in J} \sum_{(A,B,C)\in \Lmd}  (\max(A,B,C))^{-1} \|F(t_j)\|_{S} \|G(t_j)\|_{S}\|H(t_j)\|_{S} \\
&\lesssim \sum_{j\in J} \Big( \f{T}{\al}\Big)^{-\f{1}{6}+\f{1}{100}} \|F\|_{X_T}\|G\|_{X_T}\|H\|_{X_T}\\
&\lesssim \al^{\f{1}{13}} T^{-\f{1}{13}}\|F\|_{X_T}\|G\|_{X_T}\|H\|_{X_T}.
\end{split}
\end{equation*}
The proof of \eqref{h3} is similar.\\

\noindent
In the  $\cN_*^t=\cN_0^t$ case, 
\begin{equation*}
	\begin{split}
		\cA_K=& \int_0^{2\pi} \int_{I_j}   \int_{\R^3} (e^{-i\tilde{\cD}(\theta, t)}Q_AF^a) \overline{(e^{-i\tilde{\cD}(\theta, t)}Q_BF^b)}(e^{-i\tilde{\cD}(\theta, t)}Q_CF^c) \overline{(e^{-i\tilde{\cD}(\theta, t)}K)} dxdzdtd\theta,
	\end{split}
\end{equation*}
where $\tilde{\cD}(\theta, t):=\theta\cH_0-t\partial_z^2/2$.
We decompose the integral interval and change the variable for $\theta$, and apply Corollary \ref{bi3}. Then we have
\begin{equation*}
	\begin{split}
			\cA_K &=  \int_0^{\f{\pi}{2}} +\int_{\f{\pi}{2}}^{\f{3\pi}{2}}+ \int_{\f{3\pi}{2}}^{2\pi}   \\
			&\qquad  \Big[ \int_{I_j \times \R^3} (e^{-i\tilde{\cD}(\theta, t)}Q_AF^a) \overline{(e^{-i\tilde{\cD}(\theta, t)}Q_BF^b)}(e^{-i\tilde{\cD}(\theta, t)}Q_CF^c) \overline{(e^{-i\tilde{\cD}(\theta, t)}K)} dxdzdt \Big] d\theta \\
			&= \int_0^{\f{\pi}{2}} +\int_{\f{3\pi}{2}}^{2\pi}     \\
			&\qquad  \Big[  \int_{I_j \times \R^3}(e^{-i\tilde{\cD}(\theta, t)}Q_AF^a) \overline{(e^{-i\tilde{\cD}(\theta, t)}Q_BF^b)}(e^{-i\tilde{\cD}(\theta, t)}Q_CF^c) \overline{(e^{-i\tilde{\cD}(\theta, t)}K)} dxdz dt  \Big]d\theta \\
				&\quad + \int_{-\f{\pi}{2}}^{\f{\pi}{2}} \Big[ \int_{I_j \times \R^3}  (e^{-i\tilde{\cD}(\theta, t)}Q_A e^{-i\pi\cH_0} F^a) \overline{(e^{-i\tilde{\cD}(\theta, t)}Q_B  e^{-i\pi\cH_0}F^b)}\\
			&\qquad \qquad \qquad \qquad \qquad \times (e^{-i\tilde{\cD}(\theta, t)}Q_C  e^{-i\pi\cH_0}F^c) \overline{(e^{-i\tilde{\cD}(\theta, t)}  e^{-i\pi\cH_0}K)} 
			dxdz dt \Big] d\theta \\
			&\lesssim \int_{-\f{\pi}{2}}^{\f{\pi}{2}} \Big(1+\tan^2\Big(\f{\theta}{2}\Big)\Big)^\h d\theta \times (\max(A,B,C))^{-1}\min_{\sigma \in \fS^3}\|F^{\sigma(a)}\|_{L_\bfx^2} \|F^{\sigma(b)}\|_{S}\|F^{\sigma(c)}\|_{S} \\
			&\lesssim (\max(A,B,C))^{-1}\min_{\sigma \in \fS^3}\|F^{\sigma(a)}\|_{L_\bfx^2} \|F^{\sigma(b)}\|_{S}\|F^{\sigma(c)}\|_{S}.
	\end{split}	
\end{equation*}

\end{proof}
To end the proof of the above lemma, we used the following result.  This is a simple modification of \cite[Lemma 4.4]{MS}.\\

\begin{lemma}\label{bi2}
Denote by $Q_{\lmd}$ and $Q_\mu$ the frequency localization in $z$. Assume that $\lmd\ge10\mu\ge1$ and that $u(t)=e^{-itD_0^\ep}u_0$ and $v(t)=e^{-itD_0^\ep}v_0$, where $u_0$, $v_0 \in L_z^2\Sigma_x^2(\R^3)$. Then we have the bound 
\begin{equation}
\begin{split}
	\|(Q_{\lmd}u)(Q_{\mu}v)\|_{L_{t,\bfx}^2(\R\times\R^3)} \lesssim \lmd^{-\f{1}{2}} \min\{\|u_0\|_{L_{\bfx}^2} \|v_0\|_{L_z^2\Sigma_x^2}, \|v_0\|_{L_{\bfx}^2} \|u_0\|_{L_z^2\Sigma_x^2}\}.
\end{split}
\end{equation}
\end{lemma}

\begin{proof}
	Let $A=\bigcup_{n\in\Z}\{(2n-\h)\pi\ep^2, (2n+\h\pi)\ep^2\}$, $\R=A\cup A+\pi\ep^2$. By the translation invariance and the unitarity of the flow of $\{e^{it\cH_0/\ep^2}\}_{t\in\R}$  on $L_x^2$ and $\Sigma_x^2$, it is sufficient to to prove \eqref{bi2} on $A\times \R^3$. 
	Let $f \in L^2(\R^d)$ and $H:=-\Delta_x+\om|x|^2(\om>0, x\in\R^d)$, and denote $\tilde{u}(t,\cdot)=e^{it\Delta_x/2}f$ and $\tilde{v}(t,\cdot)=e^{-itH/2}f$. Then the lens transform (cf.~\cite{pseudo}) gives
	\[ \tilde{v}(t,x) = (1+\tan^2(\om t))^{\f{d}{4}}e^{-i\f{\om}{2}|x|^2\tan(\om t)} \tilde{u}\Big(\f{\tan(\om t)}{\om}, \sqrt{1+\tan^2(\om t)} x\Big).\]
	Hence for $t\in A$, we have
\begin{equation*}
	\begin{split}
	&(e^{-it\cH_0/\ep^2}e^{it\partial_z^2/2}Q_\lmd u_0)(\bfx)(e^{-it\cH_0/\ep^2}e^{it\partial_z^2/2}Q_\mu v_0)(\bfx)\\
	&=\Big(1+\tan^2\Big(\f{t}{2\ep^2}\Big)\Big)\Big(e^{i\tan(\f{t}{2\ep^2})\Delta_x}e^{it\partial_z^2/2}Q_\lmd u_0\Big(\sqrt{1+\tan\Big(\f{t}{2\ep^2}\Big)}x,z\Big)\Big) 	\\
	&\quad  \times \Big(e^{i\tan(\f{t}{2\ep^2})\Delta_x}e^{it\partial_z^2/2}Q_\mu v_0\Big(\sqrt{1+\tan\Big(\f{t}{2\ep^2}\Big)}x,z\Big)\Big)e^{-\f{i}{2}|x|^2\tan(\f{t}{2\ep^2})}.
	\end{split}
\end{equation*}
Since $|1+\tan(\f{t}{2\ep^2})|\le 2$ holds  for any $t\in A$, by taking $L_x^2$ norm on both sides of the above equality, changing variables in $x$ and applying Plancherel's theorem, we have
\begin{equation*}
	\begin{split}
		&\|(e^{it\cD_0^\ep}Q_\lmd u_0)( e^{it\cD_0^\ep}Q_\mu v_0)\|_{L_x^2}\\
		&\lesssim\Big(1+\tan^2\Big(\f{t}{2\ep^2}\Big)\Big)^{\h}\|(e^{it\partial_z^2/2}e^{i\tan(\f{t}{2\ep^2})\Delta_x}Q_\lmd u_0)(e^{it\partial_z^2/2}e^{i\tan(\f{t}{2\ep^2})\Delta_x}Q_\mu v_0)\|_{L_x^2}\\
		&\lesssim \Big\|\int_{\R_\eta^2}(e^{i\tan(\f{t}{2\ep^2})(|\eta|^2+|\xi-\eta|^2)}) (e^{it\partial_z^2/2}\cF_x[Q_\lmd u_0](\eta,z)) (e^{it\partial_z^2/2}\cF_x[Q_\mu v_0](\xi-\eta,z)) d\eta \Big\|_{L_\xi^2} \\
		&\le	\Big\|\int_{\R_\eta^2}|(e^{it\partial_z^2/2}\cF_x[Q_\lmd u_0](\eta,z)) (e^{it\partial_z^2/2}\cF_x[Q_\mu v_0](\xi-\eta,z))| d\eta \Big\|_{L_\xi^2}	
	\end{split}
\end{equation*}
Taking $L_{t,z}^2(A\times\R)$ norm, using Minkowski's inequality and \cite[Lemma 4.3]{MS}, we have
\begin{equation*}
	\begin{split}
	&\|(e^{it\cD_0^\ep}Q_\lmd u_0)( e^{it\cD_0^\ep}Q_\mu v_0)\|_{L_x^2}\\
	&\lesssim \lmd^{-\h} \Big\| \int_{\R_\eta^2}\|\cF_z[Q_\lmd u_0](\eta,z)\|_{L_z^2}  \|\cF_z[Q_\mu v_0](\xi-\eta,z)\|_{L_z^2} d\eta  \Big\|_{L_\xi^2} \\
	&\lesssim \lmd^{-\h} \|u_0\|_{L_\bfx^2} \|v_0\|_{L_z^2\Sigma_x^2}.
	\end{split}
\end{equation*}
Note that the last inequality holds for $\lmd^{-\h} \|v_0\|_{L_\bfx^2} \|u_0\|_{L_z^2\Sigma_x^2}$.
Replacing $u_0$, $v_0$ with  $e^{-i\pi\cH_0}e^{i\pi\ep^2\partial_z^2/2}u_0$, $e^{-i\pi\cH_0}e^{i\pi\ep^2\partial_z^2/2}v_0$ respectively, we obtain the same estimate on $(A+\pi\ep^2)\times\R^3$.
\end{proof}

\begin{cor}\label{bi3}
	Suppose $Q_{\lmd}$, $Q_\mu$, assumptions for $\lmd$, $\mu$, $u_0$ and $v_0$ are the same as in Lemma \ref{bi2}. Then for  $u(\theta, t)=e^{-i\theta\cH_0+it\partial_z^2/2}u_0$ and $v(\theta,t)=e^{-i\theta\cH_0+it\partial_z^2/2}v_0$, we have the bound 
	\begin{equation}
		\begin{split}
			\|(Q_{\lmd}u)(\theta, \cdot)&(Q_{\mu}v)(\theta, \cdot)\|_{L_{t,\bfx}^2(\R\times\R^3)} \\
			&\lesssim \Big(1+\tan^2\Big(\f{\theta}{2}\Big)\Big)^\h \lmd^{-\f{1}{2}} \min\{\|u_0\|_{L_{\bfx}^2} \|v_0\|_{L_z^2\Sigma_x^2}, \|v_0\|_{L_{\bfx}^2} \|u_0\|_{L_z^2\Sigma_x^2}\}.
		\end{split}
	\end{equation}
\end{cor}
\quad
\subsection{The fast Oscillations}
We decompose $\cN^t$ as
\begin{equation}\label{tildeN}
	\cN^t[F,G,H]=\cN_0^t[F,G,H]+\tilde{\cN^t}[F,G,H]
\end{equation} 
with (recall the notation \eqref{nlpr2} and \eqref{I})
\begin{equation*}
	\begin{split}
		 \tilde{\cN^t}[F,G,H]&=\sum_{\om\neq 0}\sum_{(p,q,r,s)\in\Gamma_\om} \Pi_p \cN^t[F_q,G_r,H_s] \\
		 &=\sum_{\om\neq 0} e^{i\f{\om}{2\ep^2}t}\sum_{(p,q,r,s)\in\Gamma_\om} \Pi_p \cI^t[F_q,G_r,H_s].
	\end{split} 
\end{equation*} 
\begin{lemma}\label{fast}
Fix $\ep\in(0,1)$. Let  $T\in [1,T^*]$ and $\al\in (0,1]$. Assume that $F,G,H:\R\to S$ satisfy \eqref{asX} and
\begin{equation}\label{suppl}
 F=Q_{\le(T/\al)^{\f{1}{6}}}F, \quad G=Q_{\le(T/\al)^{\f{1}{6}}}G, \quad H=Q_{\le(T/\al)^{\f{1}{6}}}H .
 \end{equation}
Then we can write
\[\tilde{\cN^t}[F(t),G(t),H(t)]=\tilde{\cE_1^t}[F(t),G(t),H(t)]+\cE_2^t[F(t),G(t),H(t)]=:\tilde{\cE_1^t}(t)+ \cE_2^t(t)\]
and the following estimates hold uniformly in $T\in[1, T^*]$, $\al$ and $\ep$ {\rm:}
\[ T^{\f{17}{16}}\sup_{T/4\le t \le T^*}\|\tilde{\cE_1}(t)\|_S \lesssim \ep^{2}\al^{-1}+\al^{\f{1}{13}}, \qquad T^{\f{1}{2}}\sup_{T/4\le t \le T^*}\|\cE_3(t)\|_S \lesssim \ep^2. \]
where $\cE_2(t)=\partial_t\cE_3(t)$. Assuming that \eqref{asX+} holds instead of \eqref{asX}, we have
\[ T^{\f{17}{16}}\sup_{T/4\le t \le T^*}\|\tilde{\cE_1}(t)\|_{S^+} \lesssim \ep^{2}\al^{-1} +\al^{\f{1}{13}}, \qquad T^{\f{1}{2}}\sup_{T/4\le t \le T^*}\|\cE_3(t)\|_{S^+} \lesssim \ep^2. \]

\end{lemma}
\begin{proof}
We decompose $\tilde{\cN^t}$ as follows:
\begin{equation*}
	\begin{split}
	\cF_z \tilde{\cN^t}[F^a,F^b,F^c](x,\z)&= \sum_{\om\neq 0} e^{i\f{\om}{2\ep^2}t}\sum_{(p,q,r,s)\in\Gamma_\om} \Pi_p\Big(\cO_1^t[F_q^a,F_r^b,F_s^c](x,\z)+\cO_2^t[F_q^a,F_r^b,F_s^c](x,\z) \Big) \\
	&=: \cJ_1 +  \cJ_2.
\end{split}
\end{equation*}
where
\begin{equation}\label{ot1}
\begin{split}
	\cO_1^t[f^a,f^b,f^c](\z):=&\int_{\R^2}e^{it\e\kappa}(1-\varphi(\al t^{\f{1}{6}}\e \ka ))\hat{f^a}(\z-\e)\overline{\hat{f^b}(\z-\e-\kappa) } \hat{f^c}(\z-\kappa)d\e d\kappa,
\end{split}
\end{equation}
\begin{equation}
\begin{split}
	\cO_2^t[f^a,f^b,f^c](\z):=&\int_{\R^2}e^{it\e\kappa}\varphi(\al t^{\f{1}{6}}\e \ka ) \hat{f^a}(\z-\e)\overline{\hat{f^b}(\z-\e-\kappa) } \hat{f^c}(\z-\kappa)d\e d\kappa.
\end{split}
\end{equation}

\noindent
\underline{Estimate of $\cJ_1$} \quad  We have the relation
\begin{equation}\label{oibp}
\begin{split}
&e^{i\f{\om}{2\ep^2}t}\cO_2^t[f^a,f^b,f^c] \\
&=\partial_t\Big( \f{e^{i(\om/2\ep^2)t}-1}{i\om/2\ep^2} \cO_2^t[f^a,f^b,f^c]\Big) - \f{e^{i(\om/2\ep^2)t}-1}{i\om/2\ep^2} (\partial_t\cO_2^t)[f^a,f^b,f^c] \\
&\quad- \f{e^{i(\om/2\ep^2)t}-1}{i\om/2\ep^2} \Big(\cO_2^t[\partial_tf^a,f^b,f^c] +  \cO_2^t[f^a,\partial_tf^b,f^c] + \cO_2^t[f^a,f^b,\partial_tf^c] \Big)  ,
\end{split}
\end{equation}
where
\begin{equation*}
\begin{split}
 (\partial_t\cO_2^t)[f^a,f^b,f^c]:=\int_{\R^2}&\partial_t(e^{it\e\kappa}\varphi(\al t^{\f{1}{6}}\e \ka ))\hat{f^a}(\z-\e)\overline{\hat{f^b}(\z-\e-\kappa) } \hat{f^c}(\z-\kappa)d\e d\kappa.
\end{split}
\end{equation*}
Let $\lfloor \cdot \rfloor$ be floor function, then
\begin{equation*}
\begin{split}
 &\f{e^{i(\om/2\ep^2)t}-1}{i\om/2\ep^2} \cO_2^t[f^a,f^b,f^c] =\int_{2\pi\ep^2 \lfloor \f{t}{2\pi\ep^2}\rfloor }^{t}e^{i\f{\om}{2\ep^2}\tau} \cO_2^t[f^a,f^b,f^c]d\tau .
\end{split}
\end{equation*}
holds. The Similar identity holds for $\partial_t \cO_2^t$. Hence we have 
\begin{equation}\label{eqo2}
\begin{split}
\cJ_2 =\partial_t(\cF_z\cE_3)&-\int_{2\pi\ep^2 \lfloor \f{t}{2\pi\ep^2}\rfloor }^t  \sum_{\om\neq 0} e^{i\f{\om}{2\ep^2}\tau}\sum_{(p,q,r,s)\in\Gamma_\om}\Pi_{p}\Big( (\partial_t \cO_2^t)[F_q^a,F_r^b,F_s^c] \\
&+\cO_{2}^t[\partial_tF_q^a,F_r^b,F_s^c] + \cO_{2}^t[F_q^a,\partial_tF_r^b,F_s^c] + \cO_{2}^t[F_q^a,F_r^b,\partial_tF_s^c] \Big) d\tau
\end{split}
\end{equation}
where $\cE_3=\cE_3(t,x,z)$ is defined by
\begin{equation}\label{E3}
\begin{split}
\cF_z\cE_3(t,x,\z)&:=\int_{2\pi\ep^2 \lfloor \f{t}{2\pi\ep^2}\rfloor }^{t}  \sum_{\om\neq 0} e^{i\f{\om}{2\ep^2}\tau}\sum_{(p,q,r,s)\in\Gamma_\om} \Pi_p \cO_{2}^t[F_q^a,F_r^b,F_s^c](x,\z)d\tau.
\end{split}
\end{equation} 
We now estimate the contribution of each term of the right hand side of \eqref{eqo2}.
First, we treat $\cF_z\cE_3$. By \eqref{suppl} and the definition of $\cO_{2}^t$, we can define the multiplier $m$ (See \cite[Lemma 2.3]{MS}) that appears in $\cO_{2}^t$ by
\[ m(\e,\ka):=\varphi(\al t^{\f{1}{6}}\e \ka)\varphi((10T/\al)^{-\f{1}{6}}\e) \varphi((10T/\al)^{-\f{1}{6}}\ka). \]
Then we have the bound $\|\cF_{\e, \ka}m\|_{L^1(\R^2)} \lesssim (1+t)^{\delta_0}$, where $\delta_0$ is an arbitrarily small constant. To prove this, we use the same argument as  \cite[Remark 3.5]{MST}. First, by changing variables and a straightforward calculation, we have
\[ \|\cF_{\e,\ka}m\|_{L^1(\R^2)}=\|I(S)\|_{L^1(\R^2)} \]
where
\[ I(S,y_1,y_2)=\int_{\R^2} e^{iy_1\e}e^{iy_2\ka}\varphi(S\e\ka)\varphi(\e)\varphi(\ka)d\e d\ka, \qquad S \simeq \al^\f{2}{3}T^{\f{1}{2}} \]
Then it holds uniformly in $S\in(0,\infty)$ that
\begin{equation*}
\begin{split}
&|(1+y_1+y_2)I(S,y)|\lesssim 1 , \qquad  |y_1y_2I(S,y)| \lesssim 1+\log(1+S), \\
&(y_1^2+y_2^2+y_1^2y_2^2)|I(S,y)|\lesssim 1+S\log(1+S)
\end{split}
\end{equation*}
and by interpolation, for any small $\delta_0>0$ there exists some $k>\f{1}{2}$ such that 
\[ |I(S,y)|\lesssim (1+S)^{\delta_0}(1+y_1^2)^k(1+y_2^2)^k.\]
We have the aforementioned bound.\\

\noindent
Hence we apply \cite[Lemma 2.3]{MS} with $(p,q,r,s)=(2,2,\infty, \infty)$ and get, for  $t\ge T/4$,
\begin{equation}
\begin{split}\label{eqo2tt0}
&\|\cO_2^t[f^a,f^b,f^c]\|_{L_\z}^2 = \|\cF_{\z}\cO_2^t[f^a,f^b,f^c]\|_{L_z}^2 \\
	&\lesssim (1+|t|)^{\f{\delta}{100}}\min_{\s \in \fS^3}\| f^{\s(a)}\|_{L_z^2} \|e^{it\partial_z^2/2} f^{\s(b)}\|_{L_z^\infty} \| e^{it\partial_z^2/2} f^{\s(c)}\|_{L_z^\infty}.
\end{split}
\end{equation}
Since
\begin{equation}\label{dis}
	 \| e^{it\partial_z^2/2 } f \|_{L_z^\infty} \lesssim |t|^{-\f{1}{2}}\|f\|_{L_z^1} = |t|^{-\f{1}{2}}\|\langle z \rangle^{-\f{9}{10}}  \langle z \rangle^{\f{9}{10}}  f\|_{L_z^1}    \lesssim|t|^{-\f{1}{2}}\| \langle z \rangle^{\f{9}{10}}f \|_{L_z^2}
\end{equation}
holds, we have
\begin{equation}\label{eqo2tt}
	\|\cO_2^t[f^a,f^b,f^c]\|_{L_\z^2}	\lesssim (1+|t|)^{-1+\f{\delta}{100}}\min_{\sigma \in \fS^3}\| f^{\sigma(a)}\|_{L_z^2} \|  \langle z \rangle^{\f{9}{10}}f^{\sigma(b)}\|_{L_z^2} \| \langle z \rangle^{\f{9}{10}} f^{\sigma(c)}\|_{L_z^2} .
\end{equation} 
Let $K \in L_{x,\z}^2$. It follows from H\"{o}lder's inequality, \eqref{eqo2tt}, Minkowski's inequality and the embedding $\Sigma_x^2 \hookrightarrow L^\infty_x$ that
\begin{equation}\label{dual1}
	\begin{split}
		&\langle K,\sum_{(p,q,r,s)\in\Gamma_0}   \Pi_p\cO_2^t[F_q^a,F_r^b,F_s^c] \rangle_{L_{x,\z}^2} \le \sum_{(p,q,r,s)\in\Gamma_0}|	\langle K_p,  \cO_2^t[F_q^a,F_r^b,F_s^c]\rangle_{L_{x,\z}^2}| \\
		&\lesssim (1+|t|)^{-1+\f{\delta}{100}} \sum_{(p,q,r,s)\in\Gamma_0} \|K_p\|_{L_\bfx^2} \|F_q^a\|_{L_\bfx^2} \| \langle z \rangle^{\f{9}{10}} F_r^b\|_{L_x^\infty L_z^2}  \|  \langle z \rangle^{\f{9}{10}} F_s^c\|_{ L_x^\infty L_z^2}   \\
			&\lesssim  (1+|t|)^{-1+\f{\delta}{100}}\sum_{(p,q,r,s)\in\Gamma_0} \|K_p\|_{L_\bfx^2} \|F_q^a\|_{L_\bfx^2}    \langle r \rangle \|\langle z \rangle^{\f{9}{10}} F_r^b\|_{L_\bfx^2 }   \langle s \rangle \|  \langle z \rangle^{\f{9}{10}} F_s^c\|_{L_\bfx^2} .  \\
	\end{split}
\end{equation}
We use H\"{o}lder's inequality in $p$ and Young's inequality in $q,r,s$, then we have 
\begin{equation}\label{dual2}
	\begin{split}
	  &\sum_{(p,q,r,s)\in\Gamma_0} \|K_p\|_{L_\bfx^2} \|F_q^a\|_{L_\bfx^2}    \langle r \rangle \|\langle z \rangle^{\f{9}{10}} F_r^b\|_{L_\bfx^2 }   \langle s \rangle \|  \langle z \rangle^{\f{9}{10}} F_s^c\|_{L_\bfx^2 }   \\
	  &\lesssim \|K_p\|_{l_p^2L_\bfx^2} \Big\|\sum_{\substack{q,r,s \in \N_0\\ q-r+s=p}}  \|F_q^a\|_{L_\bfx^2}   \langle r \rangle \|\langle z \rangle^{\f{9}{10}} F_r^b\|_{L_\bfx^2 }   \langle s \rangle \|  \langle z \rangle^{\f{9}{10}} F_s^c\|_{L_\bfx^2} \Big\|_{l_p^2} \\
	  &\lesssim \|K\|_{L_\bfx^2} \|F_q^a\|_{l_q^2L_\bfx^2}  \Big\|  \langle r \rangle \|\langle z \rangle^{\f{9}{10}} F_r^b\|_{L_\bfx^2 } \Big\|_{l_r^1} \Big\|  \langle s \rangle \|  \langle z \rangle^{\f{9}{10}} F_s^c\|_{L_\bfx^2} \Big\|_{l_s^1} \\
	  &\lesssim \|K\|_{L_\bfx^2}  \|F^a\|_{L_\bfx^2} \|F^b\|_S\|F^c\|_S.  
	\end{split}
\end{equation}
The duality argument yields
\begin{equation*}
	 \Big\|\sum_{(p,q,r,s)\in\Gamma_0}   \Pi_p\cO_2^t[F_q^a,F_r^b,F_s^c] \Big\|_{L_{x,\z}^2} \lesssim (1+|t|)^{-1+\f{\delta}{100}} \min_{\sigma \in \fS^3} \|F^{\sigma(a)}\|_{L_\bfx^2} \|F^{\sigma(b)}\|_S\|F^{\sigma(c)}\|_S.
\end{equation*}
On the other hand, it holds from \eqref{eqo2tt}, H\"{o}lder's inequality,  Minkowski's inequality and the embedding $\Sigma_x^2 \hookrightarrow L^\infty_x$ that
\begin{equation}\label{ot2}
	\begin{split}
		&\|\cO_2^t[F^a,F^b,F^c]\|_{L_{x, \z}^2}		\\
		&\lesssim (1+|t|)^{-1+\f{\delta}{100}}\Big\|\min_{\sigma \in \fS^3}\| F^{\sigma(a)}\|_{L_z^2} \|  \langle z \rangle^{\f{9}{10}}F^{\sigma(b)}\|_{L_z^2} \| \langle z \rangle^{\f{9}{10}} F^{\sigma(c)}\|_{L_z^2} \Big\|_{L_x^2} \\
		&\lesssim  (1+|t|)^{-1+\f{\delta}{100}} \min_{\sigma \in \fS^3}\| F^{\sigma(a)}\|_{L_\bfx^2} \|  \langle z \rangle^{\f{9}{10}}F^{\sigma(b)}\|_{ L_x^\infty L_z^2} \| \langle z \rangle^{\f{9}{10}} F^{\sigma(c)}\|_{L_x^\infty L_z^2}  \\
		&\lesssim  (1+|t|)^{-1+\f{\delta}{100}} \min_{\sigma \in \fS^3}\| F^{\sigma(a)}\|_{L_\bfx^2} \|  F^{\sigma(b)}\|_{S} \| F^{\sigma(c)}\|_{S}.
	\end{split}
\end{equation}
Hence we get
\begin{equation}
	\begin{split}
&\|\cF_z\cE_3\|_{L_{x,\z}^2} \\
&\lesssim \ep^2 \Big( \| \cO_2^t[F^a,F^b,F^c]\|_{L_{x,\z}^2} +\Big\|\sum_{(p,q,r,s)\in\Gamma_0} \Pi_p\cO_2^t[F_q^a,F_r^b,F_s^c] \Big\|_{L_{x,\z}^2} \Big)\\
&\lesssim \ep^2 (1+|t|)^{-1+\delta/100} \min_{\sigma \in \fS^3}\|  F^{\sigma(a)}\|_{L_\bfx^2} \| F^{\sigma(b)}\|_{S} \| F^{\sigma(c)}\|_{S}
	\end{split}
\end{equation}
and the following estimates  follow from \cite[Lemma 6.1]{MS}:
\[ \|\cF_z\cE_3\|_{S} \lesssim \ep^2 (1+|t|)^{-1+\delta/100} \|  F^{a}\|_{S} \| F^{b}\|_{S} \| F^{c}\|_{S}\]
\[ \|\cF_z\cE_3\|_{S^+} \lesssim \ep^2 (1+|t|)^{-1+\delta/100} \max_{\sigma \in \fS^3}\|  F^{\sigma(a)}\|_{S^+} \| F^{\sigma(b)}\|_{S} \| F^{\sigma(c)}\|_{S}  .\]
We next treat $\partial_t \cO_2^t $. Since
\begin{equation*}
\begin{split}
	\partial_t(e^{it\e\ka}\varphi(\al t^{\f{1}{6}} \e \ka)) &= i\al^{-1}t^{-\f{1}{6}} \psi(\al t^{\f{1}{6}} \e \ka) e^{it\e\ka} +\f{1}{6}t^{-1} \tilde{\psi}(\al t^{\f{1}{6}} \e \ka) e^{it\e\ka} \\
	\psi(s) &:=s\varphi(s), \qquad \tilde{\psi}(s):=s\varphi'(s) ,\qquad  s\in \R,
\end{split}
\end{equation*}
we have
\begin{equation*}
	\begin{split}
	 \Big\| \sum_{(p,q,r,s)\in \Gamma_0}\Pi_p &\partial_t\cO_2^t[F_q^a,F_r^b,F_s^c]\Big\|_{L_{x,\z}^2} \\ &\lesssim( \al^{-1}t^{-\f{7}{6}+\f{\delta}{100}}+ t^{-2+\f{\delta}{100}})  \min_{\sigma \in \fS^3}\|  F^{\sigma(a)}\|_{L_\bfx^2} \| F^{\sigma(b)}\|_{S} \| F^{\sigma(c)}\|_{S}	
	\end{split}
\end{equation*}
and
\[ \|\partial_t\cO_2^t[F^a,F^b,F^c]\|_{L_{x,\z}^2} \lesssim(\al^{-1}t^{-\f{7}{6}+\f{\delta}{100}}+ t^{-2+\f{\delta}{100}})  \min_{\sigma \in \fS^3}\|  F^{\sigma(a)}\|_{L_\bfx^2} \| F^{\sigma(b)}\|_{S} \| F^{\sigma(c)}\|_{S}.\]
Therefore the contribution of the second term in the right hand side of \eqref{eqo2} is acceptable:
\begin{equation*}
\begin{split}
\Big\| \int_{2\pi\ep^2 \lfloor \f{t}{2\pi\ep^2}\rfloor }^t \sum_{\om\neq 0} e^{i\f{\om}{2\ep^2}\tau} \Pi_p \partial_t\cO_2^t&[F_q^a,F_r^b,F_s^c]d\tau \Big\|_{S} \\
& \lesssim\ep^2 (\al^{-1}t^{-\f{7}{6}+1/100}+ t^{-2}) \|  F^{a}\|_{S} \| F^{b}\|_{S} \| F^{c}\|_{S}. 
\end{split}
\end{equation*}

\begin{equation*}
	\begin{split}
		\Big\| \int_{2\pi\ep^2 \lfloor \f{t}{2\pi\ep^2}\rfloor }^t \sum_{\om\neq 0} e^{i\f{\om}{2\ep^2}\tau} \Pi_p& \partial_t\cO_2^t[F_q^a,F_r^b,F_s^c]d\tau \Big\|_{S^+} \\
		& \lesssim\ep^2(\al^{-1}t^{-\f{7}{6}+1/100}+ t^{-2}) \max_{\sigma\in \fS^3}\|  F^{\sigma(a)}\|_{S^+} \| F^{\sigma(b)}\|_{S} \| F^{\sigma(c)}\|_{S}. 
	\end{split}
\end{equation*}

\noindent
Finally, The terms in the last line of \eqref{eqo2} are estimated by the same argument as above and the definition of $X_{T^*}$ norm. \\

\noindent
\underline{Estimate of $\cJ_1$} \quad For $t\ge\f{T}{4}$ we show
\begin{equation}\label{ot10}
\Big\|\sum_{\om\neq 0}e^{i\f{\om}{2\ep^2}t} \sum_{(p,q,r,s)\in\Gamma_\om}\cO_1^{t}[F_q^a,F_r^b,F_s^c]\Big\|_{L_\bfx^2} \lesssim \al^{\f{1}{13}}T^{-\f{17}{16}}  \min_{\sigma \in \fS^3}\|  F^{\sigma(a)}\|_{L_\bfx^2} \| F^{\sigma(b)}\|_{S} \| F^{\sigma(c)}\|_{S}.
\end{equation}
We first prove
\begin{equation}\label{ot11}
	\|\cO_1^{t}[f^a,f^b,f^c]\|_{L_{\z}^2} \lesssim \al^{\f{1}{13}}T^{-\f{17}{16}} \min_{\sigma \in \fS^3} \|  f^{\sigma(a)} \|_{L_z^2} \| \langle z \rangle ^{\f{9}{10}}f^{\sigma(b)} \|_{L_z^2} \|  \langle z \rangle ^{\f{9}{10}} f^{\sigma(c)}\|_{L_z^2}.
\end{equation}
Note that $\cF_z {\cI^t}=\cO_1^t+\cO_2^t$ holds and $\cI^t$ is estimated by H\"{o}lder in $z$ as
\begin{equation*}
	\begin{split}
		\|\cI^t[f^a,f^b,f^c]\|_{L_z^2} &\le \min_{\sigma \in \fS^3} \|f^{\sigma(a)}\|_{L_z^2} \|e^{it\partial_z^2/2}f^{\sigma(b)}\|_{L_z^\infty} \|e^{it\partial_z^2/2}f^{\sigma(c)}\|_{L_z^\infty}.
	\end{split}
\end{equation*}
So the same estimate as \eqref{eqo2tt0} holds for $\cO_1^t$. 
We denote $f=f^a$, $g=f^b$, $h=f^c$ and decompose
\[ g=g_1+g_2, \qquad g_1(z):=\varphi\Big(\f{\al z}{T^{\f{1}{6}}}\Big)g(z)\]
\[h=h_1+h_2, \qquad h_1(z):=\varphi\Big(\f{\al z}{T^{\f{1}{6}}}\Big)h(z)\]
where $\varphi \in C_0^\infty(\R)$, $\varphi(z)=1$ when $|z|\le1$ and $\varphi(z)=0$ when $|z|>2$. By estimating like \eqref{dis}, we can check 
\begin{equation}\label{inf1}
	\|e^{it\partial_z^2/2}f\|_{L_z^\infty} \lesssim_\beta R^{-\beta}|t|^{-\f{1}{2}}\|\langle z\rangle^{\f{9}{10}} f\|_{L_z^2}, \qquad \text{if } \text{supp}(f) \subset \{|z|\ge R \} 
\end{equation}
\begin{equation}\label{inf2}
	\|e^{it\partial_z^2/2}zf\|_{L_z^\infty} \lesssim_\gamma R^{1-\gamma}|t|^{-\f{1}{2}}\|\langle z \rangle^{\f{9}{10}} f\|_{L_z^2}, \qquad \text{if } \text{supp}(f) \subset \{|z|\le R\} 
\end{equation}
 hold for all $\beta, \gamma \in (0,\f{2}{5})$. By \eqref{inf1} we get
\begin{equation*}
 \|\cO_1^{t}[f,g_2,h]\|_{L_z^2} \lesssim T^{-1-\f{\beta}{6}+\f{\delta}{100}} \al^{\beta} \|  f \|_{L_z^2} \| \langle z \rangle^{\f{9}{10}} g\|_{L_z^2} \|  \langle z \rangle^{\f{9}{10}} h\|_{L_z^2}	
\end{equation*}
and the same estimate for $\cO_1^{t}[f,g_1,h_2]$. Take $\beta=\f{2}{5}-\f{1}{100}$.
Hence \eqref{ot11} is reduced to
\begin{equation}
		\|\cO_1^{t}[f,g_1,h_1]\|_{L_\z^2} \lesssim \al^{\f{1}{13}} T^{-\f{17}{16}} \|  f \|_{L_z^2} \| \langle z \rangle^{\f{9}{10}} g \|_{L_z^2} \|  \langle z \rangle^{\f{9}{10}}h\|_{L_z^2}.
\end{equation}
We integrate by part $\cO_1^t$ in the $\ka$ integral \eqref{ot1}:
\begin{equation}
\begin{split}
	\cO_1^t[f,g,h]:=&\int_{\R^2}(-it\e)^{-1}e^{it\e\kappa}\partial_\ka\Big[(1-\varphi(\al t^{\f{1}{6}}\e \ka ))\times \\
&\qquad \hat{f}(\z-\e)\overline{\hat{g}(\z-\e-\kappa) } \hat{h}(\z-\kappa)\Big]d\e d\kappa,
\end{split}
\end{equation}
Since
\[ (t\e)^{-1} \partial_\ka(1-\varphi(\al t^{\f{1}{6}} \e \ka ))= \al t^{-\f{5}{6}}  \varphi'(\al t^{\f{1}{6}} \e \ka ), \]
the contribution of the term when the $\ka$-derivative falls on  the multiplier is bounded by $\al t^{-\f{11}{6} +\f{\delta}{100}} \|f\|_{L^2} \|\langle z \rangle^{\f{9}{10}}  g\|_{L_z^2} \|\langle z \rangle^{\f{9}{10}}h\|_{L_z^2}$.

\noindent
Suppose the $\ka$-derivative falls on $g$. we note that on the support of integration, then we have
$|\ka|\lesssim \al^{-\f{1}{6}} T^{\f{1}{6}}$
 and 
$ |\e|\gtrsim \al^{-1} t^{-\f{1}{6}}|\ka|^{-1} \gtrsim \al^{-\f{5}{6}} T^{-\f{1}{3}}$. 
Especially, there exists $c>0$ such that $ 2\al t^{\f{1}{6}} |\e \ka|  \le c|\e|\al^{\f{5}{6}} T^{\f{1}{3}}   $. Then 
\begin{equation*}
\begin{split}
&\Big\|(it)^{-1}\int_{\R^2}\e^{-1}e^{it\e\kappa} (1-\varphi(\al t^{\f{1}{6}}\e \ka )) \hat{f}(\z-\e)\overline{\hat{(zg_1)}(\z-\e-\kappa) } \hat{h_1}(\z-\kappa) d\e d\kappa \Big\|_{L_{\z}^2} \\
\lesssim & \al^{\f{5}{6} } T^{-\f{2}{3}}\Big\| \int_{\R^2}   \f{1-\varphi(c \al^{\f{5}{6}}T^{\f{1}{3}}\e )}{c\al^{\f{5}{6}}T^{\f{1}{3}} \e }  (1-\varphi(\al t^{\f{1}{6}} \ka ))\e^{-1}e^{it\e\kappa} \times \\
&\qquad \qquad \qquad \qquad \qquad \qquad \qquad  \hat{f}(\z-\e)\overline{\hat{(zg_1)}(\z-\e-\kappa) } \hat{h_1}(\z-\kappa) d\e d\kappa \Big\|_{L_\z^2} \\
\lesssim &  \al^{\f{5}{6}} T^{-\f{2}{3} +\f{\delta}{100}} \|f\|_{L_z^2} \|e^{it\partial_z^2/2}zg_1\|_{L_z^\infty} \|e^{it\partial_z^2/2}h_1\|_{L_z^\infty} \\
\lesssim &  \al^{\f{1}{13}} T^{-\f{17}{16}} \|f\|_{L_z^2} \|\langle z \rangle^{\f{9}{10}}g\|_{L_z^2} \|\langle z \rangle^{\f{9}{10}}h\|_{L_z^2} .
\end{split}
\end{equation*}
In the above estimate, we use the same argument in \eqref{eqo2tt} with
\[m(\e, \ka) = \f{1-\varphi(c \al^{\f{5}{6}}T^{\f{1}{3}}\e )}{c\al^{\f{5}{6}}T^{\f{1}{3}} \e }(1-\varphi(\al t^{\f{1}{6}}\e \ka))\varphi((10T/\al)^{-\f{1}{6}}\e) \varphi((10T/\al)^{-\f{1}{6}}\ka)\]
and \eqref{inf2} with $\gamma=\f{2}{5}-\f{1}{100}$. Changing variables in $\e$, $\ka$, and iterating the same argument, we obtain \eqref{ot11}. Furthermore, the duality argument as \eqref{dual1} and \eqref{dual2} yields
\begin{equation}\label{ot100}
	\Big\| \sum_{(p,q,r,s)\in\Gamma_0}\cO_1^{t}[F_q^a,F_r^b,F_s^c]\Big\|_{L_\bfx^2} \lesssim \al^{\f{1}{13}}T^{-\f{17}{16}}  \min_{\sigma \in \fS^3}\|  F^{\sigma(a)}\|_{L_\bfx^2} \| F^{\sigma(b)}\|_{S} \| F^{\sigma(c)}\|_{S}.
\end{equation}
On the other hand, using \eqref{ot11} and the same estimate as \eqref{ot2} we also have
\begin{equation}
	\Big\| \cO_1^{t}[F^a,F^b,F^c]\Big\|_{L_\bfx^2} \lesssim \al^{\f{1}{13}}T^{-\f{17}{16}}  \min_{\sigma \in \fS^3}\|  F^{\sigma(a)}\|_{L_\bfx^2} \| F^{\sigma(b)}\|_{S} \| F^{\sigma(c)}\|_{S}.
\end{equation}
Hence we get \eqref{ot10}.\\

Now we conclude the proof of Lemma \ref{fast}. We define $\cE_2:=\partial_t\cE_3$, where $\cE_3$ is given by \eqref{E3}, and  $\tilde{\cE_1}$ as the sum of $\cO_1^t$ and the remaining terms in \eqref{eqo2}. Then they satisfy the claim.
\end{proof}

\subsection{Proof of Proposition \ref{nlest1} and \ref{nlest2}}
 Before to prove Proposition \ref{nlest2}, we need to consider the error between two resonant nonlinearities, $\cN_0^t$ and $\cR$. Since the error has nothing to do with $\ep$, it is sufficient to follow \cite[Section 4.3]{MS}. 
 \begin{prop}[\cite{MS} Proposition 4.6]
 Assume $N\ge7$ and let $t\ge1$. Then we have
 \begin{equation}\label{NR}
 	\|\cN_0^t[F,G,H]-\f{\pi}{t}\cR[F,G,H]\|_{Z} \lesssim (1+t)^{-\f{33}{32}}\|F\|_S\|G\|_S\|H\|_S
 \end{equation}
\begin{equation}\label{NR2}
	\|\cN_0^t[F,G,H]-\f{\pi}{t}\cR[F,G,H]\|_{S} \lesssim (1+t)^{-\f{33}{32}}\|F\|_{S+}\|G\|_{S+}\|H\|_{S+}.
\end{equation}
 \end{prop}
In view of \cite[Proposition 4.6]{MS}, the decay rates in time depends on $\delta$. But in fact they obtained \eqref{NR} and \eqref{NR2} in the proof.

\begin{proof}[Proof of Proposition \ref{nlest1}]
Fix $\ep\in (0,1)$. We decompose $\cN^t$ as
\begin{equation}\label{decNf}
\begin{split}
\cN^t[F,G,H]=& \cN_0^t[F,G,H]  \\
 &+\sum_{\substack{A,B,C \in 2^{\N_0} \\ \max(A,B,C)\ge(T/\ep)^{1/6}}} \Big(\cN^t[Q_AF,Q_BG,Q_CH] - \cN_0^t[Q_AF,Q_BG,Q_CH] \Big)\\ 
 &+	\tilde{\cN^t}[Q_{\le (T/\ep)^{\f{1}{6}}}F,Q_{\le (T/\ep)^{\f{1}{6}}} G,Q_{\le (T/\ep)^{\f{1}{6}}} H].
\end{split}
\end{equation}
Apply Lemma \ref{high} and Lemma \ref{fast} as $\al=\ep$.
\end{proof}
\begin{proof}[Proof of Proposition \ref{nlest2}]
Refer to \cite[Section 4.4]{MS} and \cite[Section 3.3]{MST}.
We use Lemma \ref{high} and \ref{fast} as $\al=1$.
\end{proof}

\section{Proof of the main theorems}\label{proof}

\subsection{Modified wave operators}\label{mwo}

Before we prove the main theorems, we present the a priori estimates for $U^\ep$ and $W$.  
\begin{prop}\label{apriori}
	There exists an absolute constant $\al>0$ such that any initial data $\psi_0 \in S^+$ satisfying 
	\begin{equation}
		\|\psi_0\|_{S^+}\le \al
	\end{equation}
generates a global solution to \eqref{Uep} and it holds for any $T>0$ that
\begin{equation}
	\|U^{\ep}\|_{X_T^+}\le 2\al.
\end{equation}
 A similar statement holds for the solution to \eqref{pr} (in Section \ref{prs}).
\end{prop}

\begin{proof}
	See \cite[Proposition 6.2]{MST}. In the proof we apply Proposition \ref{nlest2}. Because Proposition \ref{nlest2} has some better estimates than  \cite[Proposition 4.1]{MS}, we can give $\al$ as an absolute constant.
\end{proof}
\quad\\
\begin{proof}[Proof of Theorem \ref{main1}]
We define a complete metric space
\[\fA\label{key}:=\{ F\in C^1([1,\infty),S ): \|F\|_{\fA}\le \al \}\]
\begin{equation}
\begin{split}
	\|F\|_{\fA}:=\sup_{t\ge 1}\{& \ep^{-\f{1}{20}}(1+|t|)^{\f{1}{16}}\|F(t)\|_{L^2} + \ep^{-\f{1}{20}}(1+|t|)^{\f{1}{20}+10\delta}\|F(t)\|_Z \\
	&+ \ep^{-\f{1}{20}}(1+|t|)^{\f{1}{20}+9\delta}\|F(t)\|_S  +(1+|t|)^{1-10\delta}\|\partial_tF(t)\|_S \\
	&+ \ep^{-\f{1}{20}}(1+|t|)^{\f{1}{20}}\|F(t)\|_{S^+}  +(1+|t|)^{1-10\delta}\|\partial_tF(t)\|_{S^+} \}
\end{split}
\end{equation}
where $\al$ is a positive constant.  We show the map
\[\Phi(F)(t):= i\int_{t}^\infty \Big(\cN^s[F+W,F+W,F+W]-\cN_0^s[W,W,W]\Big)ds \]
is on $\fA$ and a contraction for some $\al$. \\

\noindent
 First, if we take $\al>0$ so that satisfying Proposition \ref{apriori}, the solution to \eqref{pr} with $W(0)=\psi_0$ exists globally and satisfies for all $t>0$ that
 \begin{equation}\label{apw0}
 \begin{split}
 	\|W(t)\|_Z &+(1+|t|)^{-\delta}\|W(t)\|_S + (1+|t|)^{1-3\delta}\|\partial_t W(t)\|_S \\
 	&+(1+|t|)^{-5\delta}\|W(t)\|_{S^+} + (1+|t|)^{1-7\delta}\|\partial_t W(t)\|_{S^+} \lesssim \al.
 \end{split}
 \end{equation}
we decompose
\begin{equation*}
\begin{split}
	&\cN^s[F+W,F+W,F+W]-\cN_0^s[W,W,W] \\
	&=  \cN^s[F,F,F] + 	2\cN^s[F,F,W] +	\cN^s[F,W,F] \\
	&\quad + 2\cN^s[F,W,W] + \cN^s[W,F,W] +  \tilde{\cN^s}[W,W,W]
\end{split}
\end{equation*}
and denote
\[ \cA[F,W]:=  \cN^s[F,F,F] + 	2\cN^s[F,F,W] +	\cN^s[F,W,F] \]
\[\cB[F,W]:= 2\cN^s[F,W,W] + \cN^s[W,F,W].\]
It is sufficient to show for any $F$, $F_1$, $F_2\in \fA$ that  
\begin{equation}\label{estA}
	\Big\|\int_t^\infty \cA[F(s),W(s)]ds\Big\|_{\fA} \lesssim  \al^3
\end{equation}
\begin{equation}\label{estB}
	\Big\|\int_t^\infty \cB[F(s),W(s)]ds\Big\|_{\fA} \lesssim \al^3
\end{equation}
\begin{equation}\label{estA2}
	\Big\|\int_t^\infty (\cA[F_1(s),W(s)] -  \cA[F_2(s),W(s)])ds\Big\|_{\fA} \lesssim  \al^2 \|F_1-F_2\|_{\fA}
\end{equation}
and show
\begin{equation}\label{estN}
	\Big\|\int_t^\infty \tilde{\cN^t}[W(s),W(s),W(s)]ds\Big\|_{\fA} \lesssim \al^3.
\end{equation}
We first prove \eqref{estB}. The same argument can be applied to \eqref{estA} and \eqref{estA2}. By Proposition \ref{nlest1}, \eqref{N01}, \eqref{apw0} and the definition of $Z_t^\delta$, $Y_t$ and $\fA$ norms, we have
\begin{equation}\label{SL2}
	\begin{split}
		& \Big\|\int_t^\infty \cB[F,W]ds\Big\|_{L_\bfx^2} \\
		&\lesssim \int_{t}^\infty (\| \cN_0^s[F,W,W] \|_{L_\bfx^2} + \|\cN_0^s[W,F,W] \|_{L_\bfx^2} )ds  \\
		&\qquad \qquad  + \Big\|\int_t^\infty (2\tilde{\cN}^s[F,W,W]  +  \tilde{\cN}^s[W,F,W] ) ds \Big\|_{S} \\
		&\lesssim \int_t^\infty (1+s)^{-1}  \|F\|_{L_\bfx^2}\|W\|_{Z_s^\delta}^2 ds + \ep^{\f{1}{20}}(1+|t|)^{-\f{1}{16}} \al^2\|F\|_{Y_t} \\
		&\lesssim \ep^{\f{1}{20}}(1+|t|)^{-\f{1}{16}} \al^3.
	\end{split}
\end{equation}
Similarly, 
\begin{equation*}
\begin{split}
		& \Big\|\int_t^\infty \cB[F,W]ds\Big\|_{Z} \lesssim \int_{t}^\infty ( \| \cN_0^s[F,W,W] \|_Z + \|\cN_0^s[W,F,W] \|_Z)ds   + \ep^{\f{1}{20}}(1+|t|)^{-\f{1}{16}}\al^3.
			\end{split}
	\end{equation*}
By interpolation in \cite[Lemma 2.1 (2)]{MS}:
\begin{equation}\label{interpo}
	\|F\|_{Z}\lesssim \|F\|_{L^2}^{\f{1}{8}} \|F\|_{S}^{\f{7}{8}}, 
\end{equation}
\eqref{N01} and \eqref{N03}, we have
\begin{equation*}
	\begin{split}
		\|\cN_0^s[F,W,W] \|_Z&\lesssim  \|\cN_0^s[F,W,W] \|_{L_\bfx^2}^{\f{1}{8}}  \|\cN_0^s[F,W,W] \|_S^{\f{7}{8}} \\
		&\lesssim (1+s)^{-1} (\|F\|_{L_\bfx^2}\|W\|_{Z_s^\delta}^2)^{\f{1}{8}} (\|W\|_{Z_s^\delta}^2   \|F\|_S + \|W\|_{Z_s^\delta} \|W\|_{S}  \|F\|_{Z_s^\delta} )^{\f{7}{8}}\\
		&\lesssim \ep^{\f{1}{20}} (1+s)^{-1-\f{33}{640}} \|F\|_{\fA}\al^2 \lesssim \ep^{\f{1}{20}} (1+s)^{-\f{21}{20}-10\delta} \al^3.
	\end{split}
\end{equation*}
Hence we have
\begin{equation}\label{SZ}
	\begin{split}
			\Big\|\int_t^\infty \cB[F,W]ds\Big\|_{Z}\lesssim \ep^{\f{1}{20}}(1+|t|)^{-\f{1}{20}-10\delta} \al^3.
\end{split}
\end{equation}
Next, by \eqref{N03} and \cite[Lemma 2.2]{MS}, we have
\begin{equation}\label{SS}
	\begin{split}
		& \Big\|\int_t^\infty \cB[F,W]ds\Big\|_{S} \\
		&\lesssim \int_{t}^\infty (\| \cN_0^s[F,W,W] \|_S + \|\cN_0^s[W,F,W] \|_S) ds  + \ep^{\f{1}{20}}(1+|t|)^{-\f{1}{16}} \|F\|_{Y_t} \al^2 \\
		&\lesssim \int_t^\infty (1+s)^{-1} (\|W\|_{Z_t^\delta}^2   \|F\|_S + \|W\|_{Z_t^\delta} \|W\|_{S}  \|F\|_{Z_t^\delta} )ds + \ep^{\f{1}{20}}(1+|t|)^{-\f{1}{16}} \|F\|_{Y_t}\al^2 \\
		&\lesssim \ep^{\f{1}{20}}(1+|t|)^{-\f{1}{20}-9\delta} \al^3
	\end{split}
\end{equation}
and
\begin{equation}\label{SdS}
\|\cB[F,W]\|_{S} \lesssim (1+|t|^{-1})\|F\|_S \|W\|_S^2 \lesssim (1+|t|)^{-\f{21}{20}-7\delta} \al^3.
\end{equation}
Furthermore, by \eqref{N04} we also have
\begin{equation}\label{SS+}
	\begin{split}
		& \Big\|\int_t^\infty \cB[F,W]ds\Big\|_{S^+} \\
		&\lesssim \int_{t}^\infty (\| \cN_0^s[F,W,W] \|_{S^+} + \|\cN_0^s[W,F,W] \|_{S^+}) ds  + \ep^{\f{1}{20}}(1+|t|)^{-\f{1}{16}} \|F\|_{Y_t^+}\al^2 \\
		&\lesssim \int_t^\infty \Big[(1+s)^{-1} (\|W\|_{Z_t^\delta}^2   \|F\|_{S^+} + \|W\|_{Z_t^\delta} \|W\|_{S^+}  \|F\|_{Z_t^\delta} ) \\
		& \qquad +(1+s)^{-1+2\delta} (\|W\|_{Z_t^\delta} \|W\|_S   \|F\|_S +  \|W\|_{S}^2  \|F\|_{Z_t^\delta} ) \Big]ds + \ep^{\f{1}{20}}(1+|t|)^{-\f{1}{16}} \al^3 \\
		&\lesssim \ep^{\f{1}{20}}(1+|t|)^{-\f{1}{20}} \al^3
	\end{split}
\end{equation}
and
\begin{equation}\label{SdS+}
\begin{split}
	\|\cB[F,W]\|_{S^+} &\lesssim (1+|t|^{-1})(\|F\|_{S^+} \|W\|_S^2 + \|F\|_S \|W\|_{S^+} \|W\|_S)\lesssim (1+|t|)^{-\f{21}{20}} \al^3.	
\end{split}
\end{equation}

\noindent
We finally consider \eqref{estN}. From Proposition \ref{nlest1}, we have

\begin{equation}\label{tN1}
	\Big\| \int_t^\infty \tilde{\cN^s}[W(s),W(s),W(s)]ds \Big\|_{S^+} \lesssim \ep^\f{1}{20} (1+t)^{-\f{1}{16}} \|W\|_{Y_t^+}^3 \lesssim  \ep^\f{1}{20} (1+t)^{-\f{1}{16}} \al^3.
\end{equation}
By \cite[Lemma 2.2]{MS}, \eqref{N04} and \eqref{apriori}, we have
\begin{equation}\label{tN2}
	\begin{split}
		&\|  \tilde{\cN^s}[W(t),W(t),W(t)]\|_{S^+} \\
		&\lesssim  (1+|t|)^{-1} \|W(t)\|_{S^+}\|W(t)\|_{S}^2 + (1+|t|)^{-1+2\delta} \|W(t)\|_{S}^2\|W(t)\|_{Z_t^\delta}\\
		&\lesssim (1+|t|)^{-1+7\delta}\al^3.	
	\end{split}
\end{equation}
By the estimates from \eqref{SL2} to \eqref{tN2} we obtain \eqref{estB}.\\

If we choose $\al$ sufficiently small, we construct the desired $F\in C([1,\infty),S^+). $ Let $U^\ep(t):=W(t)-F(t)$, we have \eqref{main1-1} and \eqref{main1-2} for $t\ge1$. To extend \eqref{main1-2} to $[0,1]$, we use Lemma \ref{3}.

\end{proof}

\subsection{Strong magnetic confinement limit}

\begin{proof}[Proof of Theorem \ref{main2}]
The statement (i) can be proved by changing $U^\ep$ and $W$ in the proof of \ref{main1}. That is, we define  $\fA$ as above and consider the map
\[\Psi^\ep(F)(t):= i\int_{t}^\infty \Big(\cN_0^s[F-U^\ep,F-U^\ep,F-U^\ep]+\cN^s[U^\ep,U^\ep,U^\ep]\Big)ds, \]
and use 
\begin{equation*}
\begin{split}
		\|U^\ep(t)\|_Z &+(1+|t|)^{-\delta}\|U^\ep(t)\|_S + (1+|t|)^{1-3\delta}\|\partial_t U^\ep(t)\|_S  \\
		&+(1+|t|)^{-5\delta}\| U^\ep(t) \|_{S^+} + (1+|t|)^{1-7\delta}\|\partial_t U^\ep(t)\|_{S^+} \lesssim \al,
\end{split}
\end{equation*}
where $U^\ep$ is the solution to \eqref{Uep} with the data $\psi_0$.\\

We now show the statement (ii). By using Lemma \ref{apriori} again, a solution to \eqref{pr} $W\in C([0,\infty),S^+)$ with the data $\psi_0$ satisfies
\begin{equation}\label{apw2}
	\begin{split}
		\|W(t)\|_Z &+(1+|t|)^{-\delta}\|W(t)\|_S + (1+|t|)^{1-3\delta}\|\partial_t W(t)\|_S \\
		&+(1+|t|)^{-5\delta}\|W(t)\|_{S^+} + (1+|t|)^{1-7\delta}\|\partial_t W(t)\|_{S^+} \lesssim \al.
	\end{split}
\end{equation} 
Here,  we use the following propositions.

\begin{prop}\label{th3s} 	Let $\phi \in  C([0,\Tmax), S^+)$ be a maximal solution to (\ref{lmt0})  with  $\phi(0)=\psi_0\in S^+$. Then for any $T\in (0,\Tmax)$ there exists a solution to (\ref{ep0}) $\psiep \in  C([0, T], S^+)$ with  $\psi(0)=\psi_0$ for all sufficiently small $\ep>0$, which satisfies
		\begin{equation}
			\lim_{\ep \to +0}\| e^{it\cH/\ep^2}\psiep-\phi\|_{L^\infty ([0,T],S^+)}=0.
		\end{equation}
		Furthermore, there exist $ \ep_1=\ep_1(\|\psi_0\|_{S^+},T)>0$ and $C=C(\|\psi_0\|_{\Sigma},T)>0$ such that for any $\ep \in (0 , \ep_1]$,
		\begin{equation}
			\|e^{it\cH/\ep^2}\psiep-\phi\|_{L^\infty([0,T],S)}\le C\ep .
		\end{equation}
\end{prop}
This proposition is an application of \cite[Theorem 1.5]{K} with $\sigma=1$ and $V\equiv 0$, and follows from the next lemma.

\begin{lemma}\label{3}
	Suppose $\psiep_0, \phi_0 \in S^+$ satisfy
	\begin{equation}\label{hyiv}
		\lim_{\ep \to +0 }\| \psiep_0 -\phi_0 \|_{S^+} = 0. 
	\end{equation}
	Let $\psiep \in C([0,\Tmax^\ep),S^+)$ be a maximal solution to (\ref{ep}) with the data $\psiep_0$ and $\phi \in C([0,\Tmax),S^+)$ be a maximal solution to (\ref{lmt}) with the data $\phi_0$. Then for any small $\beta>0$, there exist $\ep_0=\ep_0(\phi_0, \beta)$, $T=T(\| \phi_0 \|_{S^+}) \in (0, \inf_{ \ep\in(0,\ep_0] }\{1,\Tmax^\ep, \Tmax \})$ and $C=C(\| \phi_0 \|_{S^+})>0$ such that
	
	\begin{equation}\label{3-1}
		\| e^{it\cH/\ep^2}\psiep-\phi \|_{L^\infty([0,T], S^+)}\le C(\beta +\| \psiep_0 - \phi_0 \|_{S^+})
	\end{equation}
	holds for all $\ep\in (0,\ep_0]$. Moreover, there exists $\ep_1=\ep_1(\| \phi_0 \|_{S})>0$ such that
	\begin{equation}\label{3-2}
		\| e^{it\cH/\ep^2}\psiep-\phi \|_{L^\infty([0,T],S)}\le C(\ep +\| \psiep_0 - \phi_0 \|_{S}) 
	\end{equation}
	holds for all $\ep \in (0,\ep_1]$.
\end{lemma}

\begin{proof}
It is sufficient to follow most of the proof of \cite[Proposition 5.1]{K},  but we need some modifications.\\

First, we approximate $\phi$ by the solution to \eqref{lmt0} with the initial data
\[\phi_{0,N}:= P_{\le N}\phi_0, \]
where $P_{\le N}$ is in Section \ref{preli} in this paper. 
We need this approximation to estimate the terms corresponding to $I_2$ and $I_3$ in the proof of \cite[Proposition5.1]{K}.\\

Second, we do not need $L_t^4\Sigma_{x,z}^{4,2}$ norm used in the proof of \cite[Proposition 5.1]{K}. By \eqref{normeq}, the commutator $[z, e^{it\partial_z^2/2}]=-it\partial_z$, H\"{o}lder's inequality, Minkowski's inequality and the embedding $\Sigma_x^2 \hookrightarrow L_x^\infty$ and $H_z^1 \hookrightarrow L_z^\infty$, we have
\begin{equation*}
	\begin{split}
		\|A_1 \|_{L_t^\infty([0,T], S^+)}&\lesssim  \| u_n^\ep\|_{L_t^\infty ([0,T],S^+)}( \| \psiep \|_{L_t^\infty ([0,T],S^+)}^2 +\| \tpsi_n \|_{L_t^\infty ([0,T],S^+) }^2 )T.
	\end{split}
\end{equation*}
Here, we use the same notation as the proof of \cite[Proposition 5.1]{K}. 
Note that we use $[z, e^{it\partial_z^2/2}]=-it\partial_z$ instead of \cite[Lemma 2.1]{K} throughout our proof. \\

Finally, it hold that
\[ \|\partial_z^k\phi_{0,N}\|_{S^+}\lesssim N^k\|\phi_0\|_{S^+} \quad \text{for} \quad k\in \N_0\]
and
\[  \|\partial_z^k\phi_{0,N}\|_{S}\lesssim \|\phi_0\|_{S^+} \quad \text{for}\quad  k\le8, \qquad  \| \phi_0-\phi_{0,N}\|_{S} \lesssim N^{-8}\|\phi_0\|_{S^+}. \]
 \quad\\

By the above modifications, we have \eqref{3-1} and \eqref{3-2} from the same argument as the proof of \cite[Proposition 5.1]{K}.
\end{proof}

\begin{prop}\label{APW}
	Suppose solutions to \eqref{pr} $W_1(t), W_2(t) \in C([0,\infty), S^+)$ satisfy \eqref{apw2} for some $\al>0$ and an arbitrarily small $\delta>0$. then the following estimates hold for $t\in [1,\infty)$\rm{:}
	\begin{equation}\label{dW}
		\|W_1(t)-W_2(t)\|_S \lesssim  \|W_1(1)-W_2(1)\|_S t^{C \al^2}
	\end{equation} 
	\begin{equation}\label{dW+}
		\|W_1(t)-W_2(t)\|_{S^+} \lesssim \|W_1(1)-W_2(1)\|_{S^+} t^{C \al^2}.
	\end{equation}
\end{prop}

\begin{proof}
	Let $\tilde{W}_j(t):= W_j(e^t)$, $\tilde{W}_j$ solves
	\begin{equation*}
		i\partial_t \tilde{W}_j(t)=e^t\cN_0^{e^t}[\tilde{W}_j(t), \tilde{W}_j(t), \tilde{W}_j(t)]
	\end{equation*}
	\begin{equation}\label{tW}
		\tW_j(t)=\tW_j(0)-i\int_0^t e^s\cN_0^{e^s}[\tilde{W}_j(s), \tilde{W}_j(s), \tilde{W}_j(s)] ds.
	\end{equation}
Taking $L_\bfx^2$ norm and applying \eqref{N01}, we get for some $C_1>0$ that
\begin{equation}\label{L2W}
	\begin{split}
		\| \tW_1(t)-\tW_2(t) \|_{L^2}  &\le \| \tW_1(0)-\tW_2(0) \|_{L^2} \\
		&\quad + C\int_0^t \f{e^s}{1+e^s} \| \tW_1(s) - \tW_2(s) \|_{L^2} (\|\tW_1(s)\|_{Z_{e^s}^\delta}^2 + \|\tW_s(s)\|_{Z_{e^s}^\delta}^2) ds\\
		&\le \| \tW_1(0)-\tW_2(0)\|_{L^2}  +C_1\al^2 \int_{0}^t  \|  \tW_1(s)-\tW_2(s) \|_{L^2} ds.  \\
	\end{split}
\end{equation}
By Gronwall's lemma, we have 
\begin{equation}\label{L2W2}
	\|\tW_1(t)-\tW_2(t)\|_{L^2}  \le  \|\tW_1(0)-\tW_2(0)\|_{L^2} e^{C_1\al^2 t}.
\end{equation}
On the other hand, by taking $S$ norm in both sides of \eqref{tW} and applying \eqref{N03} and \eqref{interpo}, we have
\begin{equation}\label{SW}
	\begin{split}
		&\|\tW_1(t)-\tW_2(t)\|_{S} \\
		&\le \|\tW_1(0)-\tW_2(0)\|_S  +C \int_0^t  \f{e^s}{1+e^s}\Big[\|\tW_1(s)-\tW_2(s)\|_{S}(\|\tW_1\|_{Z_{e^s}^\delta}^2+ \|\tW_2\|_{Z_{e^s}^\delta}^2)  \\
		&\qquad  + \|\tW_1(s)-\tW_2(s)\|_{Z_{e^s}^{\delta}}(\|\tW_1\|_{S}+ \|\tW_2\|_{S})(\|\tW_1\|_{Z_{e^s}^\delta}+ \|\tW_2\|_{Z_{e^s}^\delta})\Big] ds \\
		&\le \|\tW_1(0)-\tW_2(0)\|_S \\
		&\qquad  +C_1\al^2 \int_0^t \Big[ \|\tW_1(s)-\tW_2(s)\|_{S} +e^{\delta s} \| \tW_1(s)-\tW_2(s)\|_{L_\bfx^2}^\f{1}{8} \| \tW_1(s)-\tW_2(s)\|_{S}^{\f{7}{8}} \Big] ds  , \\
	\end{split}
\end{equation} 
where we replace $C_1>2$ so that \eqref{SW} and 
\begin{equation}\label{interpo2}
	\|F\|_{L^2} \le C_1 \|F\|_S, \quad \forall F\in S
\end{equation}
hold. Let $F(t):= e^{-79(C_1^2\al^2+\delta)t}\|\tW_1(t)-\tW_2(t)\|_{S}$. Then there exists $T>0$ such that 
\begin{equation}\label{apSW}
	F(t)\le C_1F(0)e^{(C_1^2\al^2+\delta)t}
\end{equation}
holds for all $t\in [0,T]$. Multiplying both sides of \eqref{SW} by $e^{-79(C_1^2\al^2+\delta) t}$ and using \eqref{L2W2}, \eqref{interpo2} and  \eqref{apSW}, we have for $t\in [0,T]$ that 
\begin{equation}\label{SW2}
	\begin{split}
		 F(t)&\le F(0)  + e^{-79(C_1^2\al^2+\delta )t}C_1\al^2 \int_0^t \Big[ e^{79(C_1^2\al^2+\delta )s} F(s) \\
		&\qquad +e^{\delta s} (\| \tW_1(0)-\tW_2(0)\|_{L_\bfx^2}e^{C_1\al^2s})^{\f{1}{8}} (F(s)e^{79(C_1^2\al^2+\delta)s})^{\f{7}{8}} \Big] ds  \\
		&\le F(0)  + C_1\al^2 \int_0^t  F(s) ds + C_1^2\al^2F(0)\int_0^\infty e^{-5(C_1^2\al^2)s} ds  \\
		&\le 2F(0) + C_1\al^2 \int_0^t  F(s) ds  . \\
	\end{split}
\end{equation} 
Gronwall's lemma yields
\begin{equation}\label{apSW2}
	F(t)\le 2F(0)e^{C_1\al^2t}.
\end{equation}
Hence by the usual bootstrap argument \eqref{apSW} holds for all $t\in [0,\infty)$. 
Since $\al$ is an  absolute constant and  $\delta$ is an arbitrarily small constant, we can take $\delta$ smaller than $\al^2$, which implies \eqref{dW}. \\

We next prove \eqref{dW+} by the same strategy. 
By \eqref{N04}, $S \hookrightarrow Z$, \eqref{apSW} and $S^+ \hookrightarrow S$, we obtain for some $C_2>2$ that
\begin{equation}\label{SW+}
	\begin{split}
		&\|\tW_1(t)-\tW_2(t)\|_{S^+} \\
		&\le \|\tW_1(0)-\tW_2(0)\|_{S^+}  +C\int_0^t  \Big[\|\tW_1(s)-\tW_2(s)\|_{S^+}(\|\tW_1\|_{Z_{e^s}^{\delta}}^2+ \|\tW_2\|_{Z_{e^s}^{\delta}}^2)  \\
		&\qquad  + \|\tW_1(s)-\tW_2(s)\|_{Z_{e^s}^{\delta}}(\|\tW_1\|_{S^+}+ \|\tW_2\|_{S^+})(\|\tW_1\|_{Z_{e^s}^{\delta}}+ \|\tW_2\|_{Z_{e^s}^{\delta}}) \\
		&\qquad + e^{2\delta s}\Big(\|\tW_1(s)-\tW_2(s)\|_{S}(\|\tW_1\|_{S}+ \|\tW_2\|_{S})(\|\tW_1\|_{Z_{e^s}^{\delta}}+ \|\tW_2\|_{Z_{e^s}^{\delta}})  \\
		&\qquad  +\|\tW_1(s)-\tW_2(s)\|_{Z_{e^s}^{\delta}}(\|\tW_1\|_{S}+ \|\tW_2\|_{S})^2\Big)\Big] ds \\
		&\le \|\tW_1(0)-\tW_2(0)\|_{S^+} +C\al^2 \int_0^t \Big[  \|\tW_1(s)-\tW_2(s)\|_{S^+}  +e^{5\delta s} \| \tW_1(s)-\tW_2(s)\|_{S}    \Big] ds \\
		&\le C_2 e^{C_2\al^2t}\|\tW_1(0)-\tW_2(0)\|_{S^+} +C_2\al^2 \int_0^t  \|\tW_1(s)-\tW_2(s)\|_{S^+} ds .
	\end{split}
\end{equation} 
Let $F_+(t):=e^{-C_2\al^2t} \|\tW_1(s)-\tW_2(s)\|_{S^+}$, we have
\[  F_+(t)\le C_2F_+(0) + C_2\al^2 \int_0^t F_+(s) ds.\]
By Gronwall's lemma, we have \eqref{dW+}. 
\end{proof}
\quad\\

We return to  the proof of Theorem \ref{main2}. Applying Proposition \ref{th3s} to $\psiep=e^{-it\cD^\ep}U^\ep$ and $\phi=e^{it\partial_z^2/2}W$, it holds
\begin{equation}\label{dFW}
	\lim_{\ep\to+0}\|U^\ep-W\|_{L^\infty([0,1],S^+)}=0.
\end{equation}
From Proposition \ref{APW} and the property of $W^\ep$, we have for $t\ge1$ that 
\begin{equation*}
\begin{split}
	\|U^\ep(t)-W(t)\|_{S^+} &\le 	\|U^\ep(t)-W^\ep(t)\|_{S^+} + 	\|W^\ep(t)-W(t)\|_{S^+} \\
	& \lesssim \ep^{\f{1}{20}} (1+t)^{-\f{1}{20}} + 	\|W^\ep(1)-W(1)\|_{S^+} (1+t)^{C\al^2}\\
	 &\lesssim C_{\ep,\psi_0}(1+t)^{C\al^2} .
\end{split}
\end{equation*}
where $C_{\ep,\psi_0}:=C(\ep^{\f{1}{20}}+\|W^\ep(1)-W(1)\|_{S^+})$. From \eqref{dFW} and the property of $W^\ep$ we obtain 
\begin{equation*}
	\begin{split}
		C_{\ep,\psi_0} \lesssim \ep^{\f{1}{20}} +(\|W^\ep(1)-U^\ep(1)\|_{S^+} +\|U^\ep(1)-W(1)\|_{S^+})  \to 0 \text{\quad as \quad} \ep\to+0.
	\end{split}
\end{equation*} 
Moreover, by  Proposition \ref{APW} there exists $\ep_1=\ep_1(\al)>0$ such that it holds for all $\ep\in (0,\ep_1)$ that

\begin{equation}\label{dFW2}
	\|U^\ep-W\|_{L^\infty([0,1],S)} \lesssim \ep.
\end{equation}
Hence we also have for $t\ge1$ that
\begin{equation*}
	\begin{split}
		\|U^\ep(t)-W(t)\|_{S} &\le 	\|U^\ep(t)-W^\ep(t)\|_{S} + 	\|W^\ep(t)-W(t)\|_{S} \\
		& \lesssim \ep^{\f{1}{20}} (1+t)^{-\f{1}{20}} + 	\|W^\ep(1)-W(1)\|_{S} (1+t)^{C\al^2}  \\
		& \lesssim \ep^{\f{1}{20}} +(\|W^\ep(1)-U^\ep(1)\|_{S} +\|U^\ep(1)-W(1)\|_{S}) (1+t)^{C\al^2} \\
		&\lesssim \ep^{\f{1}{20}} (1+t)^{C\al^2}.
	\end{split}
\end{equation*}
\end{proof}
 
 \section*{Acknowledgements}
 \noindent
 The author is indebted to Kenji Nakanishi, for useful advice for this work. \\
 \noindent
 The author is supported by JST, the establishment of university fellowships towards the creation of science technology innovation, Grant Number JPMJFS2123.

\bibliographystyle{abbrv}
\footnotesize{
\bibliography{Gapprox}

\begin{thebibliography}{10}

\bibitem{Scpht}
P.~Antonelli, R.~Carles, and J.~Drumond~Silva.
\newblock Scattering for nonlinear {S}chrödinger equation under partial
  harmonic confinement.
\newblock {\em Commun. Math. Phys.}, \textbf{334}(1):367--396, 2015.

\bibitem{Gdb}
A.~H. Ardila and R.~Carles.
\newblock Global dynamics below the ground states for {NLS} under partial
  harmonic confinement.
\newblock {\em Commun. Math. Sci.}, \textbf{19}(4):993--1032, 2021.

\bibitem{N.Ben2}
N.~Ben~Abdallah, F.~Castella, F.~Delebecque-Fendt, and F.~Méhats.
\newblock The strongly confined {S}chrödinger-{P}oisson system for the
  transport of electrons in a nanowire.
\newblock {\em SIAM J. Appl. Math.}, \textbf{69}(4):1162--1173, 2009.

\bibitem{N.Ben}
N.~Ben~Abdallah, F.~Castella, and F.~M\'{e}hats.
\newblock Time averaging for the strongly confined nonlinear {S}chr\"{o}dinger
  equation, using almost-periodicity.
\newblock {\em J. Differ. Equ.}, \textbf{245}:154--200, 2008.

\bibitem{N.Ben3}
N.~Ben~Abdallah, F.~Méhats, C.~Schmeiser, and R.~M. Weishäupl.
\newblock The nonlinear {S}chrödinger equation with a strongly anisotropic
  harmonic potential.
\newblock {\em SIAM J. Math. Anal.}, \textbf{37}(1):189--199, 2005.

\bibitem{CRH}
T.~Buckmaster, P.~Germain, Z.~Hani, and J.~Shatah.
\newblock Analysis of {(CR)} in higher dimension.
\newblock {\em Int. Math. Res. Not. IMRN}, (4):1265--1280, 2019.

\bibitem{LBL}
E.~Faou, P.~Germain, and Z.~Hani.
\newblock The weakly nonlinear large-box limit of the 2{D} cubic nonlinear
  {S}chrödinger equation.
\newblock {\em J. Am. Math. Soc.}, \textbf{29}(4):915--982, 2016.

\bibitem{Fr}
R.~L. Frank, F.~Méhats, and C.~Sparber.
\newblock Averaging of nonlinear {S}chr\"{o}dinger equations with strong
  magnetic confinement.
\newblock {\em Commun. Math. Sci.}, \textbf{15}(7):1933--1945, 2017.

\bibitem{CR2}
P.~Germain, Z.~Hani, and L.~Thomann.
\newblock On the continuous resonant equation for {NLS}. {II}. {S}tatistical
  study.
\newblock {\em Anal. {PDE}}, \textbf{8}(7):1733--1756, 2015.

\bibitem{CR}
P.~Germain, Z.~Hani, and L.~Thomann.
\newblock On the continuous resonant equation for {NLS}. {I. D}eterministic
  analysis.
\newblock {\em J. Math. Pures Appl.(9)}, \textbf{105}(1):131--163, 2016.

\bibitem{BEC}
C.~Hainzl and B.~Schlein.
\newblock Dynamics of {B}ose-{E}instein condensates of fermion pairs in the low
  density limit of {BCS} theory.
\newblock {\em J. Funct. Anal.}, \textbf{265}(3):399--423, 2013.

\bibitem{MST}
Z.~Hani, B.~Pausader, N.~Tzvetkov, and N.~Visciglia.
\newblock Modified scattering for the cubic {S}chrödinger equation on product
  spaces and applications.
\newblock {\em Forum Math. Pi}, \textbf{3}(e4):63 pp, 2015.

\bibitem{MS}
Z.~Hani and L.~Thomann.
\newblock Asymptotic behavior of the nonlinear {S}chrödinger equation with
  harmonic trapping.
\newblock {\em Commun. Pure Appl. Math.}, \textbf{69}(9):1727--1776, 2016.

\bibitem{K}
J.~Kawakami.
\newblock Averaging of strong magnetic nonlinear {S}ch\"{o}dinger equations in
  energy space.
\newblock {\em preprint, https://doi.org/10.48550/arXiv.2212.06457}.

\bibitem{pseudo}
T.~Tao.
\newblock A pseudoconformal compactification of the nonlinear {S}chrödinger
  equation and applications.
\newblock {\em New York J. Math.}, \textbf{15}:265--282, 2009.

\bibitem{Yajima}
K.~Yajima and G.~Zhang.
\newblock Local smoothing property and {S}trichartz inequality for
  {S}chrödinger equations with potentials superquadratic at infinity.
\newblock {\em J. Differ. Equ.}, \textbf{202}(1):81--110, 2004.

\end{thebibliography}
}

\end{document}